\documentclass[12pt]{article}
\usepackage{amscd}
\usepackage{latexsym}
\usepackage{amstext}
\usepackage{mathrsfs}
\usepackage{verbatim}
\usepackage[all]{xy}
\usepackage{graphicx}
\usepackage{url}
\usepackage{tikz} 
\usepackage{multirow} 
\usepackage{array} 
\usepackage{float} 
\usepackage{stmaryrd}
\usepackage{amsthm}
\usepackage{amssymb}
\usepackage{amsmath}
\usepackage{mathtools}
\usepackage{arydshln}
\usepackage{enumerate}
\usepackage{nicematrix}
\usepackage{caption}
\usepackage{makecell}
\usepackage{booktabs}
\usepackage{hyperref}

\numberwithin{equation}{section}

\theoremstyle{plain} 
\newtheorem{theorem}{Theorem}[section]
\newtheorem{lemma}[theorem]{Lemma}
\newtheorem{corollary}[theorem]{Corollary}
\newtheorem{proposition}[theorem]{Proposition}

\theoremstyle{definition} 
\newtheorem{definition}[theorem]{Definition} 
\newtheorem{alg}{Algorithm}
\newtheorem{example}[theorem]{Example}

\theoremstyle{remark} 
\newtheorem{remark}[theorem]{Remark}

\begin{document}
    	\title{Characteristic polynomials and some combinatorics for finite Coxeter groups}
    	\author{ Shoumin Liu, Yuxiang Wang}
    	\date{}
    	\maketitle
    	\begin{abstract}
    		Let $W$ be a finite Coxeter group with Coxeter generating set $S=\{s_1,\ldots,s_n\}$, and $\rho$ be a complex finite dimensional representation of $W$. The characteristic polynomial of $\rho$ is defined as
    		\begin{equation*}
    			d(S,\rho)=\det[x_0I+x_1\rho(s_1)+\cdots+x_n\rho(s_n)],
    		\end{equation*}
    		where $I$ is the identity operator. In this paper, we show the existence of a combinatorics structure within $W$, and thereby prove that for any two complex finite dimensional representations $\rho_1$ and $\rho_2$ of $W$, $d(S,\rho_1)=d(S,\rho_2)$ if and only if $\rho_1 \cong \rho_2$.
    	\end{abstract}
    	
		\section{Introduction}
		\hspace{1.5em}Suppose $A_0,\ldots,A_n$ are bounded linear operators on a Hilbert space $V$. Their {\it projective joint spectrum}, introduced by R. Yang in \cite{Y}, is the set
		\begin{equation*}
			\sigma(A_1,\ldots,A_n)=\{[x_1:\cdots:x_n]\in\mathbb{C}\mathbb{P}^n:x_1A_1+\cdots+x_nA_n\ {\rm is\ not\ invertible}\}.
		\end{equation*}
		Projective spectrum has played an important role in the study of functional analysis, representation theory and spectral dynamical systems. Fruitful results have been obtained in \cite{CSZ,GY,HY,HY2,HY3,HY4}. 
		
		When V is finite dimensional, $\sigma(A_1,\ldots,A_n)$ is determined by the multivariable homogeneous characteristic polynomial
		\begin{equation*}
			f_A(x)=\det[x_1A_1+\cdots+x_nA_n].
		\end{equation*}
		
		In recent years, the spectrum $\sigma(A_1,\ldots,A_n)$ and the polynomial $f_A(x)$ were intensely investigated in representation theory of groups and Lie algebras, see \cite{CST,GLW,JL,S2}. Given a finite group $G$ with a generating set $S=\{s_1,\ldots,s_n\}$, let ${\mathcal R}(G)$ denote the set of all finite dimensional complex linear representations of $G$. For any $\rho \in {\mathcal R}(G)$, $\rho:G\rightarrow GL(V)$, the projective joint spectrum of $\rho$ is defined as
		\begin{equation*}
			D(S,\rho)=\sigma(I,\rho(s_1),\ldots,\rho(s_n)),
		\end{equation*}
		where $I$ is the identity operator on $V$. The characteristic polynomial of $\rho$ is
		\begin{equation*}
			d(S,\rho)=\det[x_0I+x_1\rho(s_1)+\cdots+x_n\rho(s_n)].
		\end{equation*}
		\par
		Let $\hat{G}=G\setminus\{1\}$, where 1 is the identity element of $G$. Let $\lambda_G$ denote the left regular representation of $G$, the polynomial $d(\hat{G},\lambda_G)$ was studied by Dedekind and Frobenius in the early research of representation theory. Let ${\it Irr}(G)$ be the set of equivalence classes of irreducible representations of $G$, and ${\it deg}\,\rho$ be the degree of $\rho$. In 1896 Frobenius proved that
		\begin{equation*}
			d(\hat{G},\lambda_G)=\prod_{\rho\in{\it Irr}(G)}(d(\hat{G},\rho))^{{\it deg}\,\rho}.
		\end{equation*}
		Besides, $d(\hat{G},\rho)$ is an irreducible polynomial for each $\rho\in{\it Irr}(G)$. The polynomial $d(\hat{G},\lambda_G)$ was later called the {\it group determinant} of $G$. For more information about it, we refer the reader to \cite{C,D,D2,D3,F,J}. Although so many good properties hold for the group determinant and the general $d(\hat{G},\rho)$, as the order of $G$ increases, the calculation and application of the polynomials become essentially inaccessible. So one tends to take a small enough generating set $S$ and expects that $D(S,\rho)$ and $d(S,\rho)$ could retain all the salient features of $D(\hat{G},\rho)$ and $d(\hat{G},\rho)$.
		\par
		Basically, given a finite group $G$ and a generating set $S$, one wonders whether the following properties hold:
		\begin{enumerate}
			\item[{\rm(a)}]
			$D(S,\lambda_G)$ determines $G$.
			\item[{\rm(b)}]
			$D(S,\rho_1)=D(S,\rho_2)$ as subschemes of $\mathbb{C}\mathbb{P}^n$ implies $\rho_1\cong\rho_2$.
			\item[{\rm(c)}]
			$D(S,\rho)$ is reduced and irreducible whenever $\rho$ is finite dimensional and irreducible.
    	\end{enumerate}
		In general, the answers are negative. Counterexamples can be found in \cite{KV,L,PS}. 
		\par 
		However, when considering $G$ as a Coxeter group and $S$ as a Coxeter generating set of $G$, some positive answer were obtained in recent studies. In \cite{CST}, authors show that, considering $D(S,\rho)$ as algebraic subschemes of $\mathbb{C}\mathbb{P}^n$, (a) holds within all Coxeter groups, (b) holds for any finite Coxeter groups of classical types. Later the paper \cite{LP} shows a generalization of (a). Suppose $G$ is taken from all Coxeter groups without infinite bonds in their Coxeter diagrams, and $D(S,\rho)$ is considered as a set, the results in \cite{LP} indicates that $D(S,\rho)$ determines $G$ for any faithful representation $\rho$ of $G$. In \cite{CST,PS2}, positive answers to some generalizations of (b) for certain infinite Coxeter groups are obtained.
		\par
		In this paper, we first present a sufficient condition for (b), which is the existence of a combinatorics structure in $G$. Then we show this structure exists for any finite Coxeter groups. Finally we show (b) holds for any finite Coxeter groups.

		It is easily seen that (b) is equivalent to the property (${\rm b}^{\prime}$): $d(S,\rho_1)=d(S,\rho_2)$ implies $\rho_1\cong\rho_2$. So we mainly use statement (${\rm b}^{\prime}$) so as to study the relation between polynomials and representations. Before we present the main theorem of this paper, let's define some notations. Let
		\begin{equation*}
			M(S)\coloneqq\{t_1t_2\cdots t_m:t_i\in S,m\in\mathbb{N}\},
		\end{equation*}
		which is a free monoid generated by alphabet $S$. Let a word $w\in M(S)$, the {\it signature} of $w$ is ${\rm sig}(w)\coloneqq(a_1,\ldots,a_n)\in\mathbb{N}^n$ where $a_i$ is the number of times that $s_i$ appears in $w$. There is a natural homomorphism 
		\begin{equation*}
			\phi:M(S)\rightarrow G
		\end{equation*}
		where $\phi(s_i)=s_i$ for each $i$.	We use $[n]$ to denote the set $\{1,\ldots,n\}$. 
		\par 
		As the main result of this paper, we prove the following theorem.
		\begin{theorem}\label{MT}
			Let $W$ be a finite Coxeter group with $S=\{s_1,\ldots,s_n\}$ as its Coxeter generating set.
			\begin{enumerate}
				\item[{\rm(a)}]
				There is an order of all conjugacy classes of $W$: $C_1,\ldots,C_r$ and a sequence $\alpha_1,\ldots,\alpha_r\in\mathbb{N}^n$ such that
				\begin{align*}
					&\{g=\phi(w):{\rm sig}(w)=\alpha_i\}\cap C_i\neq\varnothing\quad\text{for}\ i\in[r],
					\\
					&\{g=\phi(w):{\rm sig}(w)=\alpha_i\}\cap C_j=\varnothing\quad\text{for}\ i,j\in[r]\ \text{and}\ j>i.
			    \end{align*}
			    \item[{\rm(b)}]
			    For two finite dimensional complex linear representations $\rho_1$ and $\rho_2$ for $W$, $d(S,\rho_1)=d(S,\rho_2)$ if and only if $\rho_1\cong\rho_2$.
            \end{enumerate}
		\end{theorem}
		In this paper, such a sequence $\alpha_1,\ldots,\alpha_r$ in (a) is called an {\it independent signature sequence} (ISS) with a lower triangular {\it independent signature matrix} (ISM). Basically, the main idea of the proof is to show the existence of such sequence. In \cite{CST}, authors provided a key theorem \cite[Theorem 7.5]{CST}, cited as Theorem \refeq{Tsig} in this paper, which provides a relation between the characters of the representations and combinatorics of words for a finitely generated group $G$. Based on this, it can be proven that if $G$ has such a sequence, then for any $\rho_1,\rho_2\in\mathcal{R}(G)$, the equation $d(S,\rho_1)=d(S,\rho_2)$ implies $\chi_{\rho_1}=\chi_{\rho_2}$ where $\chi$ denotes the character, see Proposition \refeq{keyprop}. Besides, we prove that if each of $G_1,\ldots,G_k$ has such a sequence, then so does $G_1\times\cdots\times G_k$, see Corollary \refeq{dpc}. Thus the crucial step is to show that each irreducible finite Coxeter group has an independent signature sequence with a lower triangular independent signature matrix, which is one of the core components of this paper.
		\par
		The structure of the paper is sketched as follows. In Section 2, we recall some necessary conceptions and basic results about the Coxeter system. In Section 3, we define independent signature sequence and present some properties about it. In Section 4 and 5, we show that each irreducible finite Coxeter group has an independent signature sequence with a lower triangular independent signature matrix. For classical types, our proof is inspired by the proof in \cite[Section 7]{CST} where the authors proved that for types ${\rm A,B\ \text{and}\ D}$, $d(S,\rho_1)=d(S,\rho_2)$ if and only if $\rho_1\cong\rho_2$. We refine its proof and obtain the existence of an independent signature sequence with a lower triangular independent signature matrix for each classical type. For exceptional types, we start with their conjugacy classes and reduce the problem to a theorem about cuspidal classes (a special type of conjugacy classes), namely Theorem \refeq{cmp}. Cuspidal classes of exceptional types were listed by M. Geck and G. Pfeiffer in \cite[Appendix B]{GP}. The proof of Theorem \refeq{cmp} is presented in Section 6 and is largely based on the results obtained by running the algorithm we developed in MATLAB. In Section 7, we complete the proof of Theorem \refeq{MT}, establish a semigroup isomorphism between representations and their characteristic polynomials as a corollary, and provide two examples for $W({\rm H_3})$ and $W({\rm A_n})$.

		\section{The Coxeter system}\label{sec2}
		\hspace{1.5em}We first recall some basic results about the Coxeter system, which
		can be found in \cite{H} or any other books about it.
		\begin{definition}
			Let $M=(m_{ij})_{1\leq i,j\leq n}$ be a symmetric $n\times n$ matrix with entries from $\mathbb{N}\cup \{\infty\}$ such that $m_{ii}=1$ for all $i\in[n]$ and $m_{ij}>1$ whenever $i\neq j$. The {\it Coxeter group} of type $M$ is the group
				\begin{equation*}
					W(M)=\langle\{s_1,\ldots,s_n\}\ |\ \{(s_is_j)^{m_{ij}}=1,\ i,j\in [n],m_{ij}<\infty\}\rangle.
				\end{equation*}
				$S=\{s_1,\ldots,s_n\}$ is called the {\it Coxeter generating set}. If no confusion is imminent, we write $W$ instead of $W(M)$. The pair $(W,S)$ is called the {\it Coxeter system} of type $M$.	
		\end{definition}

		Fix a Coxeter matrix $M=(m_{ij})_{i,j\in[n]}$. Let $V$ be a real vector space with basis $(e_i)_{i\in[n]}$. Denote by $B_M$, or just $B$ if $M$ is clear from the context, the symmetric bilinear form on $V$ determined by
		\begin{equation*}
			B(e_i,e_j)=-2\cos(\pi/m_{ij}).
		\end{equation*}
		The linear transformations $\sigma_i$ ($i\in[n]$) is given by
		\begin{equation*}
			\sigma_i(x)=x-B(x,e_i)e_i\quad(x\in V).
		\end{equation*}
		
    	\begin{theorem}[{\cite[Section 5.3, 5.4]{H}}]\label{refrep}
    		Let $(W,S)$ be a Coxeter system.
    		\begin{enumerate}
    			\item[{\rm(a)}]
    			The mapping $w\mapsto\sigma_w$ given by $\sigma_w=\sigma_{r_1}\cdots\sigma_{r_q}$ whenever $w=s_{r_1}\cdots s_{r_q} \\ \in W$ with $r_j\in [n]$ ($j\in[q]$) defines a linear representation of $W$ on $V$, which is called the reflection representation (or the geometric representation) of $W$, denoted by $\rho_{\it ref}$.
    			\item[{\rm(b)}]
    			The reflection representation of $W$ is faithful.
    		\end{enumerate}
    	\end{theorem}
    	
		\begin{definition}
			The Coxeter matrix $M=(m_{ij})_{1\leq i,j\leq n}$ is often described by a labelled graph $\Gamma(M)$ whose vertex set is $[n]$ and in which two nodes $i$ and $j$ are joined by an edge labelled $m_{ij}$ if $m_{ij}>2$. If $m_{ij}=3$, then the label 3 of of the edge $\{i,j\}$ is often omitted. If $m_{ij}=4$, then instead of the label 4 at the edge $\{i,j\}$ one often draws a double bond. This labelled graph is called the {\it Coxeter diagram} of $M$. We say that the Coxeter system $(W,S)$ is irreducible if the Coxeter diagram $\Gamma$ is connected. 
		\end{definition}
		
		\begin{definition}
			Let $(W,S)$ be a Coxeter system, $J\subseteq S$ be any subset and consider the subgroup $W_J\coloneqq\langle J\rangle\leqslant W$. Then $W_J$ is called a parabolic subgroup of $W$. We call $W_J$ a proper parabolic subgroup if $J\neq S$.
		\end{definition}
		
		\begin{proposition}[{\cite[Section 2.2]{H}}]
			Let Coxeter system $(W,S)$ have Coxeter diagram $\Gamma$, with connected components $\Gamma_1,\ldots,\Gamma_r$, and let $S_1,\ldots,S_r$ be the corresponding subsets of $S$. Then $W$ is the direct product of the parabolic subgroups $W_{S_1},\ldots,W_{S_r}$, and each Coxeter system $(W_{S_i},S_i)$ is irreducible.
		\end{proposition}
		
		\begin{theorem}[{\cite[Section 2.4, 2.7]{H}}]
			An irreducible Coxeter group is finite if and only if its Coxeter diagram occurs in Table \refeq{table1} (see Section 4) or Table \refeq{table2} (see Section 5), where the types in the former are called {\it classical types} while the types in the latter are called {\it exceptional types}.
		\end{theorem}
		
		\begin{corollary}
			A Coxeter group $G$ is finite if and only if $G\cong W(M_1)\times W(M_2)\times\cdots\times W(M_k)$ where each $M_i$ occurs in Table \refeq{table1} or Table \refeq{table2}.
		\end{corollary}
	    
		\section{Independent signature sequences and their properties}
		\hspace{1.5em}Suppose $G$ is a finitely generated group with a generating set $S=\{s_1,\ldots,\\s_n\}$. The unit element of $M(S)$ is denoted by $\varepsilon$. Given a word $w\in M(S)$, the length of $w$ is the number of letters in $w$, denoted by $|w|$. Given a signature $\alpha=(a_1,\ldots,a_n)\in\mathbb{N}^n$, the length of $\alpha$ is $|\alpha|=\sum_{i=1}^n a_i$.
	    \par
		For the finite dimensional complex linear representations of $G$, we have the following theorem.
		\begin{theorem}[{\cite[Theorem 7.5]{CST}}]\label{Tsig}
			Let $G$ be any finitely generated group with generating set $S=\{s_1,\ldots,s_n\}$, and $\rho_1,\rho_2 \in {\mathcal R}(G)$. If $d(S,\rho_1)=d(S,\rho_2)$, then for each given $\alpha\in\mathbb{N}^n$, we have
			\begin{equation}\label{sig}
				\sum_{{\rm sig}(w)=\alpha} \chi_{\rho_1}(w)=\sum_{{\rm sig}(w)=\alpha} \chi_{\rho_2}(w)
			\end{equation}
			where $\chi_{\rho_i}$ denotes the character of $\rho_i$.
		\end{theorem}
		
    	Equation \eqref{Tsig} contains a lot of information for relation between two given representations. Take $\alpha={\rm sig}(s_i)$, it follows that $\chi_{\rho_1}(s_i)=\chi_{\rho_2}(s_i)$; take $\alpha={\rm sig}(s_is_j)$, it follows that $\chi_{\rho_1}(s_is_j)=\chi_{\rho_2}(s_is_j)$. We aim to show that the equality $\chi_{\rho_1}(w)=\chi_{\rho_2}(w)$ holds for any word $w$, but this is challenging to establish directly. Therefore, we develop a method to obtain the information from equation \eqref{Tsig}.
		\begin{definition}\label{evdef}
			Let $G$ be any finitely generated group with finite number of conjugacy classes $C_1,\ldots,C_r$, and $S=\{s_1,\ldots,s_n\}$ be the generating set. Given a signature $\alpha =(a_1,\ldots,a_n)$, $E_{\alpha}\coloneqq\{u\in M(S):{\rm sig}(u)=\alpha\}$. The {\it signature vector} of $\alpha$ is
				\begin{equation*}
					V_{\alpha}=(|E_{\alpha}\cap \phi^{-1}(C_1)|,\ldots,|E_{\alpha}\cap \phi^{-1}(C_r)|).
				\end{equation*}
		\end{definition}
		
		Let $\mathcal{C}_i=(\chi_{\rho_i}(C_1),\ldots,\chi_{\rho_i}(C_r))^{\rm T}$ for $i=1,2$ where $\chi_{\rho_i}(C_j)\coloneqq\chi_{\rho_i}(g)$ for an arbitrary $g\in C_j$, then we can express equation \eqref{sig} as
		\begin{equation*}
			V_{\alpha}\mathcal{C}_1=V_{\alpha}\mathcal{C}_2.
		\end{equation*}
		\begin{definition}
			Let $G,S$ and $C_i$ be as in Definition \refeq{evdef}. A sequence $\alpha_1,\ldots,\\\alpha_r$ of signatures is called an {\it independent signature sequence} (abbreviated as {\rm ISS}) if $V_{\alpha_1},\ldots,V_{\alpha_r}$ is linearly independent. Suppose $\alpha_1,\ldots,\alpha_r$ is an {\rm ISS}, the corresponding {\it independent signature matrix} (abbreviated as {\rm ISM}) is the matrix
				\begin{equation*}
					\left(
					\begin{array}{c}
						V_{\alpha_1}\\
						\vdots\\
						V_{\alpha_r}\\
					\end{array}      	
					\right).        	
				\end{equation*}  
		\end{definition}
		When we talk about {\rm ISM}, there are usually a clear {\rm ISS} and a clear order of conjugacy classes in the context.
		\begin{proposition}\label{keyprop}
			Let $G,S$ and $C_i$ be as in Definition \refeq{evdef}. If $G$ has an {\rm ISS}, then for any $\rho_1,\rho_2\in {\mathcal R}(G)$, $d(S,\rho_1)=d(S,\rho_2)$ implies $\chi_{\rho_1}(g)=\chi_{\rho_2}(g)$ for each $g\in G$.
		\end{proposition}
		\begin{proof}
			Let $\alpha_1,\ldots,\alpha_r$ be an {\rm ISS} of $G$. By Theorem \refeq{Tsig}, $V_{\alpha}\mathcal{C}_1=V_{\alpha}\mathcal{C}_2$ for any signature $\alpha$. Therefore
			\begin{equation*}
				\left(
				\begin{array}{c}
					V_{\alpha_1}\\
					\vdots\\
					V_{\alpha_r}\\
				\end{array}      	
				\right)        	
				\mathcal{C}_1=
				\left(
				\begin{array}{c}
					V_{\alpha_1}\\
					\vdots\\
					V_{\alpha_r}\\
				\end{array}      	
				\right)    
				\mathcal{C}_2.			
			\end{equation*}
			Since the {\rm ISM} is invertible, we have $\mathcal{C}_1=\mathcal{C}_2$, i.e. $\chi_{\rho_1}(C_i)=\chi_{\rho_2}(C_i)$ for each $i\in[r]$. Thus $\chi_{\rho_1}(g)=\chi_{\rho_2}(g)$ for each $g\in G$.
		\end{proof}
		\begin{remark}\label{cssmr}
			Suppose $G$ is a finite group. It is well known that for $\rho_1,\rho_2\in {\mathcal R}(G)$, $\rho_1\cong\rho_2$ if and only if $\chi_{\rho_1}(g)=\chi_{\rho_2}(g)$ for each $g\in G$. Therefore, if $G$ has an {\rm ISS}, then $d(S,\rho_1)=d(S,\rho_2)$ implies $\rho_1\cong\rho_2$.
		\end{remark}
		The rest of this section focuses on the signatures for the direct product of groups. The following lemma is an elementary conclusion in group theory.
		\begin{lemma}
			Let $G=G_1\times G_2$ where $G_1$ and $G_2$ are any groups. Let $\{C_{1,\alpha}:\alpha\in \Omega_1\}$ and $\{C_{2,\beta}:\beta\in \Omega_2\}$ be the sets of all conjugacy classes of $G_1$ and $G_2$, respectively. Then $\{C_{1,\alpha}\times C_{2,\beta}:(\alpha,\beta)\in\Omega_1\times\Omega_2\}$ is the set of all conjugacy classes of $G$.
		\end{lemma}
		
		\begin{lemma}\label{countlm}
				Let $G=G_1\times G_2$, where $G_1$ and $G_2$ are finitely generated groups. Generating sets of $G_1,G_2$ and $G$ are $S,T$ and $S\sqcup T$, respectively. Given $w=r_1\cdots r_n \in M(S\sqcup T)$ and $r_i\in S\sqcup T$, let
				\begin{equation*}
					r_i^{(1)}=
					\begin{cases}
						r_i & r_i\in S\\
						\varepsilon & r_i\in T
					\end{cases}
					,
					\quad
					r_i^{(2)}=
					\begin{cases}
						\varepsilon & r_i\in S\\
						r_i & r_i\in T
					\end{cases}
					.
				\end{equation*}
				Let $w^{(1)}=r_1^{(1)}r_2^{(1)}\cdots r_n^{(1)}$, $w^{(2)}=r_1^{(2)}r_2^{(2)}\cdots r_n^{(2)}$, and $f$ denote the map $w\mapsto w^{(1)}w^{(2)}$, then $|f^{-1}(w_1w_2)|=|w_1w_2|!/(|w_1|!|w_2|!)$ for any $w_1\in M(S)$, $w_2\in M(T)$.
		\end{lemma}
		\begin{proof}
			Since $|f^{-1}(w_1w_2)|$ is equal to the number of all the methods for selecting $|w_1|$ positions from $|w_1w_2|$ different positions. we have $|f^{-1}(w_1w_2)|=|w_1w_2|!/(|w_1|!|w_2|!)$.    
		\end{proof}
		\begin{theorem}\label{g1g2prop}
			Let $G=G_1\times G_2$, where $G_1$ and $G_2$ are any finitely generated groups with finite number of conjugacy classes. Generating sets of $G_1,G_2$ and $G$ are $S,T$ and $S\sqcup T$ respectively.
			
			\begin{enumerate}
				\item[{\rm(a)}]
				If both $G_1$ and $G_2$ have an {\rm ISS}, then $G$ has an {\rm ISS}.
				\item[{\rm(b)}]
				If both $G_1$ and $G_2$ have a lower triangular {\rm ISM}, then $G$ has a lower triangular {\rm ISM}.
			\end{enumerate}
		\end{theorem}
		\begin{proof}
			Let $C_{i,1},\ldots,C_{i,r_i}$ be all conjugacy classes of $G_i$ for $i=1,2$. Suppose $\alpha_1,\ldots,\alpha_{r_1}$ is an {\rm ISS} of $G_1$; $\beta_1,\ldots,\beta_{r_2}$ is an {\rm ISS} of $G_2$. Let $\alpha_i||\beta_j$ denote the signature spliced by $\alpha_i$ and $\beta_j$, for example, $(1,2)||(3,4)\coloneqq(1,2,3,4)$. Consider signature sequence 
			\begin{equation}\label{g1g2sig}
				\alpha_1||\beta_1,\ldots,\alpha_1||\beta_{r_2},\ldots,\alpha_{r_1}||\beta_1,\ldots,\alpha_{r_1}||\beta_{r_2}.
			\end{equation} 
			We will prove that it is an {\rm ISS} of $G$. For each $\alpha_i$ and $\beta_j$, we have
			\begin{equation*}
				E_{\alpha_i||\beta_j}=\bigsqcup_{\substack{u_1\in E_{\alpha_i}\\ u_2\in E_{\beta_j}}}f^{-1}(u_1u_2)
			\end{equation*}
			where $f$ is as defined in Lemma \refeq{countlm}. Let $V_{\alpha_i}=(x_{i,1},\ldots,x_{i,r_1})$ and $V_{\beta_j}=(y_{j,1},\ldots,y_{j,r_2})$. For each conjugacy class $C_{1,k}\times C_{2,l}$, it follows that
			\begin{equation*}
				\begin{aligned}
					&|E_{\alpha_i||\beta_j}\cap \phi^{-1}(C_{1,k}\times C_{2,l})|\\
					=&|\bigsqcup_{\substack{u_1\in E_{\alpha_i}\\ u_2\in E_{\beta_j}}}(f^{-1}(u_1u_2)\cap \phi^{-1}(C_{1,k}\times C_{2,l}))|
					\\
					=&\sum_{\substack{u_1\in E_{\alpha_i}\\ u_2\in E_{\beta_j}}}|f^{-1}(u_1u_2)\cap \phi^{-1}(C_{1,k}\times C_{2,l})|\\
					=&\sum_{\substack{u_1\in E_{\alpha_i}\cap \phi^{-1}(C_{1,k})\\ u_2\in E_{\beta_j}\cap \phi^{-1}(C_{2,l})}}|f^{-1}(u_1u_2)|
					\\
					=&\sum_{\substack{u_1\in E_{\alpha_i}\cap \phi^{-1}(C_{1,k})\\ u_2\in E_{\beta_j}\cap \phi^{-1}(C_{2,l})}}\frac{(|\alpha_i|+|\beta_j|)!}{|\alpha_i|!|\beta_j|!}\\
					=&|E_{\alpha_i}\cap \phi^{-1}(C_{1,k})|\cdot|E_{\beta_j}\cap \phi^{-1}(C_{2,l})|\frac{(|\alpha_i|+|\beta_j|)!}{|\alpha_i|!|\beta_j|!}
					\\
					=&\frac{(|\alpha_i|+|\beta_j|)!}{|\alpha_i|!|\beta_j|!}x_{i,k}y_{j,l}.
				\end{aligned}
			\end{equation*}
			Arrange all the conjugacy classes in $G$ as
			\begin{equation*}
				C_{1,1}\times C_{2,1},\ldots,C_{1,1}\times C_{2,r_2},\ldots,C_{1,r_1}\times C_{2,1},\ldots,C_{1,r_1}\times C_{2,r_2}.
			\end{equation*}
			Then we have
			\begin{equation*}
				V_{\alpha_i||\beta_j}= \frac{(|\alpha_i|+|\beta_j|)!}{|\alpha_i|!|\beta_j|!}(x_{i,1}y_{j,1},\ldots,x_{i,1}y_{j,r_2},\ldots,x_{i,r_1}y_{j,1},\ldots,x_{i,r_1}y_{j,r_2}).
			\end{equation*}
			Let $M_{i,j}\coloneqq(|\alpha_i|+|\beta_j|)!/(|\alpha_i|!|\beta_j|!)$, we have
			\begin{equation}\label{csmeq}
			\begin{gathered}
				\begin{pmatrix}
					V_{\alpha_1||\beta_1}\\
					\vdots\\
					V_{\alpha_1||\beta_{r_2}}\\
					\vdots\\
					V_{\alpha_{r_1}||\beta_1}\\
					\vdots\\
					V_{\alpha_{r_1}||\beta_{r_2}}\\
				\end{pmatrix}      	
				=
				\begin{pmatrix}
					M_{1,1}&~&~&~&~&~&~\\
					~&\ddots&~&~&~&~&~\\
					~&~&M_{1,r_2}&~&~&~&~\\
					~&~&~&\ddots&~&~&~\\
					~&~&~&~&M_{r_1,1}&~&~\\
					~&~&~&~&~&\ddots&~\\
					~&~&~&~&~&~&M_{r_1,r_2}\\
				\end{pmatrix}
				\boldsymbol{\cdot}\\
				\begin{pmatrix}	
					V_{\alpha_1}\\
					\vdots\\
					V_{\alpha_{r_1}}\\   	
				\end{pmatrix}			
				\bigotimes
				\begin{pmatrix}	
					V_{\beta_1}\\
					\vdots\\
					V_{\beta_{r_2}}\\    	
				\end{pmatrix}.	
			\end{gathered}
		    \end{equation}
			Then
			\begin{equation*}
				{\rm det}
				\begin{pmatrix}
					V_{\alpha_1||\beta_1}\\
					\vdots\\
					V_{\alpha_{r_1}||\beta_{r_2}}\\
				\end{pmatrix} 
				=
				(\prod_{i=1}^{r_1}\prod_{j=1}^{r_2}M_{i,j})
				{\rm det}^{r_2}
				\begin{pmatrix}	
					V_{\alpha_1}\\
					\vdots\\
					V_{\alpha_{r_1}}\\   	
				\end{pmatrix}
				{\rm det}^{r_1}
				\begin{pmatrix}	
					V_{\beta_1}\\
					\vdots\\
					V_{\beta_{r_2}}\\    	
				\end{pmatrix}			
				\neq 0.
			\end{equation*}
			Hence sequence \eqref{g1g2sig} is an {\rm ISS} of $G$.
			
			If the corresponding {\rm ISMs} of $\alpha_1,\ldots,\alpha_{r_1}$ and $\beta_1,\ldots,\beta_{r_2}$ are both lower triangular matrices, then the corresponding {\rm ISM} of \eqref{g1g2sig} is also a lower triangular matrix by equation \eqref{csmeq}.
		\end{proof}
		
		By Theorem \refeq{g1g2prop}, we immediately obtain the result for the case of  the direct product of a finite number of groups.
		
		\begin{corollary}\label{dpc}
			Let $G=G_1\times\cdots\times G_m$, where $G_1,\ldots,G_m$ are any finitely generated groups with finite number of conjugacy classes. Generating sets of $G_1,\ldots,G_m$ and $G$ are $S_1,\ldots,S_m$ and $S=\bigsqcup_{i=1}^mS_i$ respectively.
			\begin{enumerate}
				\item[{\rm(a)}]
				If each $G_i$ has an {\rm ISS}, then $G$ has an {\rm ISS}.
				\item[{\rm(b)}]
				If each $G_i$ has a lower triangular {\rm ISM}, then $G$ has a lower triangular {\rm ISM}.
			\end{enumerate}
		\end{corollary}
		\begin{remark}
			By Proposition \refeq{keyprop} and Corollary \refeq{dpc}, to prove Theorem \refeq{MT}, it suffices to show each irreducible finite Coxeter group has an {\rm ISS} with a lower triangular {\rm ISM}.
		\end{remark}
		
		\section{For classical types}
		\hspace{1.5em}In this section, we prove that each Coxeter group of classical types has an {\rm ISS} with a lower triangular {\rm ISM}. The Coxeter diagrams of classical types are as follows.
		\begin{table}[H]
			\centering
			\caption{Coxeter diagrams of classical types}
			\scalebox{1}{
				\begin{tabular}{l m{7.5cm}<{\centering}}\label{table1}
					\ \ type & diagram \\
					\hline
					\multirow{2.4}{*}{\,${\rm I_2(m)(m\geq 5)}$} &
					\multirow{3}{*}{
						{
							\begin{tikzpicture}[scale=0.8]
								\draw
								[line width=0.5pt] 
								(0,0) circle (0.1)
								(1.5,0) circle (0.1)     
								(0.1,0)--(1.4,0)
								node at (0,-0.4) {1}
								node at (1.5,-0.4) {2}
								node at (0.75,0.3) {$m$}
								;
							\end{tikzpicture}
						}
					}
					\\~&~
					\\
					\multirow{2.2}{*}{${\rm A_n(n\geq 1)}$} &
					\multirow{3}{*}{
						{
							\begin{tikzpicture}[scale=0.8]
								\draw
								[line width=0.5pt] 
								(0,0) circle (0.1)
								(1.5,0) circle (0.1)     
								(3,0) circle (0.1)
								(5,0) circle (0.1)
								(6.5,0) circle (0.1)
								(0.1,0)--(1.4,0)
								(1.6,0)--(2.9,0)
								(3.1,0)--(3.5,0)
								(4.5,0)--(4.9,0)
								(5.1,0)--(6.4,0)
								node at (0,-0.4) {1}
								node at (1.5,-0.4) {2}
								node at (3,-0.4) {3}
								node at (5,-0.4) {$n-1$}
								node at (6.5,-0.4) {$n$}
								node at (4.03,0) {$\cdots$}
								;
							\end{tikzpicture}
						}
					}
					\\~&~
					\\
					\multirow{2.2}{*}{${\rm B_n(n\geq 2)}$} &
					\multirow{3}{*}{
						{
							\begin{tikzpicture}[scale=0.8]
								\draw
								[line width=0.5pt] 
								(0,0) circle (0.1)
								(1.5,0) circle (0.1)     
								(3.5,0) circle (0.1)
								(5,0) circle (0.1)
								(6.5,0) circle (0.1)
								(0.1,0)--(1.4,0)
								(1.6,0)--(2,0)
								(3,0)--(3.4,0)
								(3.6,0)--(4.9,0)
								(5.09,-0.05)--(6.41,-0.05)
								(5.09,0.05)--(6.41,0.05)
								node at (0,-0.4) {1}
								node at (1.5,-0.4) {2}
								node at (3.5,-0.4) {$n-2$}
								node at (5,-0.4) {$n-1$}
								node at (6.5,-0.4) {$n$}
								node at (2.53,0) {$\cdots$}
								;
							\end{tikzpicture}
						}
					}
					\\~&~
					\\
					\multirow{7.1}{*}{${\rm D_n(n\geq 4)}$} &
					\multirow{5.5}{*}{
						{
							\begin{tikzpicture}[scale=0.8]
								\draw
								[line width=0.5pt] 
								(0,0) circle (0.1)
								(1.5,0) circle (0.1)     
								(3.5,0) circle (0.1)
								(5,0) circle (0.1)
								(6.5,0) circle (0.1)
								(5,1.5) circle (0.1)
								(0.1,0)--(1.4,0)
								(1.6,0)--(2,0)
								(3,0)--(3.4,0)
								(3.6,0)--(4.9,0)
								(5.1,0)--(6.4,0)
								(5,0.11)--(5,1.4)
								node at (0,-0.4) {1}
								node at (1.5,-0.4) {2}
								node at (3.5,-0.4) {$n-3$}
								node at (5,-0.4) {$n-2$}
								node at (6.5,-0.4) {$n$}
								node at (4.2,1.5) {$n-1$}
								node at (2.53,0) {$\cdots$}
								;
							\end{tikzpicture}
						}
					}
					\\~&~
					\\~&~
					\\~&~
					\\~&~
					\\
				\end{tabular}
			}
		\end{table}
		Our proof is based on the following lemma.
		\begin{lemma}\label{csslm}
			Let $\alpha_i=(a_{i,1},\ldots,a_{i,n})\in \mathbb{C}^n$, where $i\in[m]$ and $m\geq n$. Let $\gamma(\alpha_i)\coloneqq\{j:a_{i,j}\neq0\}$. If the following holds:
			\begin{enumerate}
				\item [\rm (a)]
				$|\gamma(\alpha_1)|=1$,
				\item [\rm (b)]
				$|\bigcup_{i=1}^{m}\gamma(\alpha_i)|=n$,
				\item [\rm (c)]
				$|\gamma(\alpha_k)\setminus\bigcup_{i=1}^{k-1}\gamma(\alpha_i)|=0\ {\rm or}\ 1$ for each $k\in\{2,\ldots,m\}$.
			\end{enumerate}
			Then there is a subsequence $\alpha_{p_1},\ldots,\alpha_{p_n}$ that satisfies:
			\begin{enumerate}
				\item [\rm (1)]
				$\alpha_{p_1},\ldots,\alpha_{p_n}$ is linearly independent,
				\item [\rm (2)]
				the matrix 
				\begin{equation*}
					\left(\begin{array}{c}
						\alpha_{p_1}\\
						\vdots\\
						\alpha_{p_n}\\
					\end{array}\right)
				\end{equation*}
				can be transformed into a lower triangular matrix by permuting the columns.
			\end{enumerate}
			
		\end{lemma}
		\begin{proof}
			Let
			\begin{equation*}
				A=\{\alpha_k:k\geq2, |\gamma(\alpha_k)\setminus\bigcup_{i=1}^{k-1}\gamma(\alpha_i)|=1\}\cup\{\alpha_1\}.
			\end{equation*} 
			By (a), (b) and (c), we have $|A|=n$. Let $A=\{\alpha_{p_1},\ldots,\alpha_{p_n}\}$ where $p_1<\cdots<p_n$. For each $r\in[n]$, let $\gamma(\alpha_{p_r})\setminus\bigcup_{i=1}^{p_r-1}\gamma(\alpha_i)=\{q_r\}$, we have $a_{p_r,q_r}\neq 0$ and $a_{p_i,q_r}=0$ for each $i< r$. Thus, by permuting the columns of the matrix $(a_{p_i,j})_{i,j\in[n]}$ according to $q_1,\ldots, q_n$, we obtain an invertible lower triangular matrix $(a_{p_i,q_j})_{i,j\in[n]}$. Therefore, for the sequence $\alpha_{p_1},\ldots,\alpha_{p_n}$, properties (1) and (2) hold.
		\end{proof}  
		
		\begin{theorem}\label{clsp1}
			Let $W$ be a finite Coxeter group of type ${\rm I_2(m)}$ with $S=\{g_1,g_2\}$ as its Coxeter generating set. Then $W$ has an {\rm ISS} with a lower triangular {\rm ISM}.
		\end{theorem}
		\begin{proof}
			We define a partial order on $\mathbb{N}^2$. Let $\alpha,\beta\in \mathbb{N}^2$, $\alpha=(a_1,a_2$), $\beta=(b_1,b_2)$. Let $\alpha\leq\beta$ if $(|\alpha|,a_1,a_2)\leq_{\rm lex}(|\beta|,b_1,b_2)$ where $\leq_{\rm lex}$ is the lexicographic order on $\mathbb{N}^3$. It can be seen that $\leq$ is a total order.
			\par
			Let $C_1,\ldots,C_r$ be all conjugacy classes of $W$. For each $i\in[r]$, take a word $w_i\in \phi^{-1}(C_i)$, let $\alpha_i={\rm sig}(w_i)$, let $H=\{w_1,\ldots,w_r\}$. Let
			\begin{equation*}
				A=\{\alpha\in \mathbb{N}^2:\exists i\ s.t.\ \alpha\leq\alpha_i\},
			\end{equation*}
			Clearly $|A|<\infty$, so let $A=\{\beta_1,\ldots, \beta_N\}$, where $\beta_1<\cdots<\beta_N$. Then we consider $\{V_{\beta_1},\ldots,V_{\beta_N}\}$ where $V_{\beta_i}$ defined in Definition \refeq{evdef}. Let $\gamma$ be as defined in Lemma \refeq{csslm}. Since $\beta_1=(0,0)$, we have $|\gamma(V_{\beta_1})|=1$. For each conjugacy class $C_h$, note that $w_h\in \phi^{-1}(C_h)$ and $\alpha_h={\rm sig}(w_h)\in A$, so $h\in\gamma(V_{\alpha_h})\in \bigcup_{i=1}^N\gamma(V_{\beta_i})$. For the arbitrariness of $h$, we have $|\bigcup_{i=1}^N\gamma(V_{\beta_i})|=r$. \par
			For each $k\in \{2,\ldots,N\}$, let $\beta_k=(b_1,b_2)$. If $|b_1-b_2|\geq 2$, then for each $h\in\gamma(V_{\beta_k})$, there exists a $w=t_1\cdots t_{|\beta_k|}\in \phi^{-1}(C_h)$, $t_i\in\{g_1,g_2\}$ such that ${\rm sig}(w)=\beta_k$ and $t_j=t_{j+1}$ for some $j$. Thus there is a $w'=t_1\cdots t_{j-1}t_{j+2}\cdots t_{|\beta_k|}$ such that $\phi(w')=\phi(w)$ and ${\rm sig}(w')<{\rm sig}(w)=\beta_k$. So ${\rm sig}(w')\in A$ and $h\in\gamma(V_{{\rm sig}(w')})\subseteq\bigcup_{i=1}^{k-1}\gamma(V_{\beta_i})$. For the arbitrariness of $h$, we have $|\gamma(V_{\beta_k})\setminus\bigcup_{i=1}^{k-1}\gamma(V_{\beta_i})|=0$.\par
			If $|b_1-b_2|=1$, suppose $b_2=b_1+1$. If $b_1=0$, then $|\gamma(V_{\beta_k})|=1$ and $|\gamma(V_{\beta_k})\setminus\bigcup_{i=1}^{k-1}\gamma(V_{\beta_i})|=0\ {\rm or}\ 1$. If $b_1>0$, then $g_2g_1\cdots g_2g_1g_2\in E_{\beta_k}$ and is conjugate to $g_1$ or $g_2$. For each $w=t_1\cdots t_{|\beta_k|}\in E_{\beta_k}\setminus\{g_2g_1\cdots g_2g_1g_2\}$ where $t_i\in\{g_1,g_2\}$, there is a $t_j=t_{j+1}$ for some $j$. Thus for each $h\in\gamma(V_{\beta_k})$, there are $w,w'\in M(S)$ such that ${\rm sig}(w')<{\rm sig}(w)=\beta_k$ and $\phi(w')$ is conjugate to $\phi(w)$. Hence $h\in\gamma(V_{{\rm sig}(w')})\subseteq\bigcup_{i=1}^{k-1}\gamma(V_{\beta_i})$. For the arbitrariness of $h$, we have $|\gamma(V_{\beta_k})\setminus\bigcup_{i=1}^{k-1}\gamma(V_{\beta_i})|=0$.\par
			If $b_1=b_2$, then $g_1g_2\cdots g_1g_2,g_2g_1\cdots g_2g_1\in E_{\beta_k}$. Let
			\begin{equation*}
				\tilde{E}_{\beta_k}\coloneqq E_{\beta_k}\setminus\{g_1g_2\cdots g_1g_2, g_2g_1 \cdots g_2g_1\}.
			\end{equation*}
			For each $w=t_1t_2\cdots t_{|\beta_k|} \in \tilde{E}_{\beta_k}$, there is a $t_j=t_{j+1}$ for some $j$. Let
			\begin{equation*}
				V'=(|\tilde{E}_{\beta_k}\cap \phi^{-1}(C_1)|,\ldots,|\tilde{E}_{\beta_k}\cap \phi^{-1}(C_r)|).
			\end{equation*}
			By the same argument in the previous cases, we have $|\gamma(V')\setminus\bigcup_{i=1}^{k-1}\gamma(V_{\beta_i})|=0$. Since $g_1g_2\cdots g_1g_2$ and $g_2g_1 \cdots g_2g_1$ are in the same conjugacy class, the number of conjugacy classes in $E_{\beta_k}$ is at most one more than the number of conjugacy classes in $\tilde{E}_{\beta_k}$, so $|\gamma(V_{\beta_k})\setminus\gamma(V')|=0\ {\rm or}\ 1$. Thus $|\gamma(V_{\beta_k})\setminus\bigcup_{i=1}^{k-1}\gamma(V_{\beta_i})|=0\ {\rm or}\ 1$.\par
			Consequently, by using Lemma \refeq{csslm}, there is a linearly independent subsequence $V_{\beta_{p_1}},\ldots,V_{\beta_{p_r}}$ of $V_{\beta_1},\ldots,V_{\beta_N}$, hence $\beta_{p_1},\ldots,\beta_{p_r}$ is an {\rm ISS} of $W$. Besides, with proper reordering of the conjugacy classes $C_1,\ldots,C_r$, the {\rm ISM} of $\beta_{p_1},\ldots,\beta_{p_r}$ can become a lower triangular matrix.
		\end{proof}
		
		For types ${\rm A,B\ \text{and}\ D}$, some good results are presented in \cite{CST}. We summarize them into the following proposition.

		\begin{proposition}[{\cite[Section 5, 6, 7]{CST}}]\label{ABDprop}
			Let $W$ be a finite Coxeter group of type ${\rm A_n}$, ${\rm B_n}$ or ${\rm D_{n+1}}$, and $S=\{g_1,g_2,\cdots\}$ is a set of Coxeter generators for $W$. Given a word $u\in M(S)$ and ${\rm sig}(u)=(a_1,\ldots,a_n)$ or $(a_1,\ldots,a_{n+1})$, let ${\rm ct}(u)\coloneqq(|u|,a_1,\ldots,a_{n-1})$.
			\begin{enumerate}
				\item[{\rm(a)}]
				For each word $w\in M(S)$, there exists a word $\tilde{w}$ such that $\phi(w)$ is conjugate to $\phi(\tilde{w})$ and ${\rm ct}(w)\geq_{\rm lex}{\rm ct}(\tilde{w})$ where $\leq_{\rm lex}$ is the lexicographic order on $\mathbb{N}^n$.
				\item[{\rm(b)}]
				Let words $w_1,w_2\in M(S)$, if ${\rm sig}(w_1)={\rm sig}(w_2)$ and ${\rm ct}(w_1)={\rm ct}(\tilde{w}_1)={\rm ct}(\tilde{w}_2)={\rm ct}(w_2)$, then $\phi(w_1)$ is conjugate to $\phi(w_2)$.
	    	\end{enumerate}
		\end{proposition}
		
	    In \cite{CST}, a partial order on $M(S)$ is defined as $u_1<u_2$ if ${\rm ct}(u_1)<_{\rm lex}{\rm ct}(u_2)$. Similar to the proof of Theorem \refeq{clsp1}, we consider applying Lemma \refeq{csslm} for types ${\rm A,B\ \text{and}\ D}$. The partial order and the above proposition make the proof for types ${\rm A,B\ \text{and}\ D}$ available.
		\begin{theorem}\label{clsp2}
			Let $W$ be a finite Coxeter group of type ${\rm A_n}$, ${\rm B_n}$ or ${\rm D_{n+1}}$ with $S$ as its Coxeter generating set. Then $W$ has an {\rm ISS} with a lower triangular {\rm ISM}.
		\end{theorem}
		\begin{proof}
			Let $C_1,\ldots,C_r$ be all conjugacy classes of $W$. Take a word $v_i\in \phi^{-1}(C_i)$ for each $i\in[r]$, let $H=\{v_1,\ldots,v_r\}$ and $J=\{w:\exists v_i\in H\ s.t.\ w\leq v_i\}$. Clearly we have $|J|<\infty$, so $(J,\leq)$ can extend to a totally ordered set $(J,\widetilde{\leq})$ (i.e. $w_i\leq w_j$ implies $w_i\widetilde{\leq}w_j$). So let $J=\{w_1,\ldots,w_N\}$, where $w_1\widetilde{<}\cdots\widetilde{<}w_N$. Let $\alpha_i={\rm sig}(w_i)$ for each $i\in[r]$.\par
			Since $w_1=\varepsilon$, it follows that $|\gamma(V_{\alpha_1})|=1$. For each $h\in[r]$, we have $v_h\in \phi^{-1}(C_h)$ and $v_h\in J$. Let $\alpha_k={\rm sig}(v_h)$, then $h\in\gamma(V_{\alpha_k})\subseteq\bigcup_{i=1}^N\gamma(V_{\alpha_i})$. For the arbitrariness of $h$, we have $|\bigcup_{i=1}^N\gamma(V_{\alpha_i})|=r$.
			\par
			For each $k\in\{2,\ldots,N\}$, by Proposition \refeq{ABDprop} part (a), we divide $E_{\alpha_k}$ into two parts:
			\begin{equation*}
				\hat{E}_{\alpha_k}\coloneqq E_{\alpha_k}\cap\{u:{\rm ct}(\tilde{u})<_{\rm lex}{\rm ct}(u)\}
			\end{equation*}
			and
			\begin{equation*}
				\bar{E}_{\alpha_k}\coloneqq E_{\alpha_k}\cap\{u:{\rm ct}(\tilde{u})={\rm ct}(u)\}.
			\end{equation*}
			Let
			\begin{equation*}
				V'=(|\hat{E}_{\alpha_k}\cap \phi^{-1}(C_1)|,\ldots,|\hat{E}_{\alpha_k}\cap \phi^{-1}(C_r)|).
			\end{equation*}	
			For each $h\in\gamma(V')$, there exists $u_0\in \hat{E}_{\alpha_k}$ and $u_0\in \phi^{-1}(C_h)$. Since ${\rm ct}(\tilde{u}_0)<_{\rm lex}{\rm ct}(u_0)={\rm ct}(w_k)$, we have $\tilde{u}_0<w_k$ and thus $\tilde{u}_0\in J$. Let $\tilde{u}_0=w_l$, ${\rm sig}(\tilde{u}_0)=\alpha_l$, then $w_l\widetilde{<}w_k$ and then $l<k$. Since $\phi(\tilde{u}_0)$ and $\phi(u_0)$ are in the same conjugacy class, $h\in \gamma(V_{\alpha_l})\subseteq\bigcup_{i=1}^{k-1}\gamma(V_{\alpha_i})$. For the arbitrariness of $h$, we have $\gamma(V')\subseteq\bigcup_{i=1}^{k-1}\gamma(V_{\alpha_i})$.
			\par
			Now consider words in $\bar{E}_{\alpha_k}$. Let $u_1,u_2\in \bar{E}_{\alpha_k}$, we have ${\rm sig}(u_1)={\rm sig}(u_2)$ and ${\rm ct}(u_1)={\rm ct}(\tilde{u}_1)={\rm ct}(\tilde{u}_2)={\rm ct}(u_2)$, so $\phi(u_1)$ and $\phi(u_2)$ are in the same conjugacy class by Proposition \refeq{ABDprop} part (b). Thus the number of conjugacy classes in $E_{\alpha_k}$ is at most one more number than the number of conjugacy classes in $\hat{E}_{\alpha_k}$. So we have $|\gamma(V_{\alpha_k})\setminus\gamma(V')|=0\ {\rm or}\ 1$, and $|\gamma(V_{\alpha_k})\setminus\bigcup_{i=1}^{k-1}\gamma(V_{\alpha_i})|=0\ {\rm or}\ 1$ since $\gamma(V')\subseteq\bigcup_{i=1}^{k-1}\gamma(V_{\alpha_i})$.\par 
			Therefore, by using Lemma \refeq{csslm}, there is a linearly independent subsequence $V_{\alpha_{p_1}},\ldots,V_{\alpha_{p_r}}$ of $V_{\alpha_1},\ldots,V_{\alpha_N}$, hence $\alpha_{p_1},\ldots,\alpha_{p_r}$ is an {\rm ISS} of $W$. Besides, with proper reordering of the conjugacy classes $C_1,\ldots,C_r$, the {\rm ISM} of $\alpha_{p_1},\ldots,\alpha_{p_r}$ can become a lower triangular matrix.	
		\end{proof}
		\begin{remark}
			The idea for the proof of Theorem \refeq{clsp2} is inspired by the proof in \cite[Section 7]{CST}, where the authors proved that for types ${\rm A,B\ \text{and}\ D}$, $d(S,\rho_1)=d(S,\rho_2)$ if and only if $\rho_1\cong\rho_2$. We refined its proof and obtained a stronger conclusion. Theorem \refeq{clsp1}, \refeq{clsp2} together with the results in the following several sections form the foundation for proving Theorem \refeq{MT}.
		\end{remark}
		
		\section{For exceptional types}
		\hspace{1.5em}In this section, we prove that each Coxeter group of exceptional types has an {\rm ISS} with a lower triangular {\rm ISM}. The Coxeter diagrams of exceptional types are as follows.
		\begin{table}[H]
			\centering
			\caption{Coxeter diagrams of exceptional types}
			\scalebox{1}{
				\begin{tabular}{l m{8cm}<{\centering}}\label{table2}
					type & diagram \\
					\hline
					\multirow{2.1}{*}{\ \ ${\rm H_3}$} &
					\multirow{2.4}{*}{
						{
							\begin{tikzpicture}[scale=0.8]
								\draw
								[line width=0.5pt] 
								(0,0) circle (0.1)
								(1.5,0) circle (0.1)
								(3,0) circle (0.1)       
								(0.1,0)--(1.4,0)
								(1.6,0)--(2.9,0)
								node at (0,-0.4) {1}
								node at (1.5,-0.4) {2}
								node at (3,-0.4) {3}
								node at (0.75,0.3) {$5$}
								;
							\end{tikzpicture}
						}
					}
					\\~&~
					\\
					\multirow{2.4}{*}{\ \ ${\rm H_4}$} &
					\multirow{3.0}{*}{
						{
							\begin{tikzpicture}[scale=0.8]
								\draw
								[line width=0.5pt] 
								(0,0) circle (0.1)
								(1.5,0) circle (0.1)
								(3,0) circle (0.1)
								(4.5,0) circle (0.1)       
								(0.1,0)--(1.4,0)
								(1.6,0)--(2.9,0)
								(3.1,0)--(4.4,0)
								node at (0,-0.4) {1}
								node at (1.5,-0.4) {2}
								node at (3,-0.4) {3}
								node at (4.5,-0.4) {4}
								node at (0.75,0.3) {$5$}
								;
							\end{tikzpicture}
						}
					}
					\\~&~
					\\
					\multirow{2.4}{*}{\ \ ${\rm F_4}$} &
					\multirow{3.4}{*}{
						{
							\begin{tikzpicture}[scale=0.8]
								\draw
								[line width=0.5pt] 
								(0,0) circle (0.1)
								(1.5,0) circle (0.1)
								(3,0) circle (0.1)
								(4.5,0) circle (0.1)       
								(0.1,0)--(1.4,0)
								(1.59,0.05)--(2.91,0.05)
								(1.59,-0.05)--(2.91,-0.05)
								(3.1,0)--(4.4,0)
								node at (0,-0.4) {1}
								node at (1.5,-0.4) {2}
								node at (3,-0.4) {3}
								node at (4.5,-0.4) {4}
								;
							\end{tikzpicture}
						}
					}
					\\
					\multirow{7.6}{*}{\ \ ${\rm E_6}$} &
					\multirow{6.1}{*}{
						{
							\begin{tikzpicture}[scale=0.8]
								\draw
								[line width=0.5pt] 
								(0,0) circle (0.1)
								(1.5,0) circle (0.1)     
								(3,0) circle (0.1)
								(4.5,0) circle (0.1)
								(6,0) circle (0.1)
								(3,1.5) circle (0.1)
								(0.1,0)--(1.4,0)
								(1.6,0)--(2.9,0)
								(3.1,0)--(4.4,0)
								(4.6,0)--(5.9,0)
								(3,0.1)--(3,1.4)
								node at (0,-0.4) {1}
								node at (1.5,-0.4) {3}
								node at (3,-0.4) {4}
								node at (4.5,-0.4) {5}
								node at (6,-0.4) {6}
								node at (2.6,1.5) {2}
								;
							\end{tikzpicture}
						}
					}
					\\~&~
					\\~&~
					\\~&~
					\\
					\multirow{7.1}{*}{\ \ ${\rm E_7}$} &
					\multirow{5.6}{*}{
						{
							\begin{tikzpicture}[scale=0.8]
								\draw
								[line width=0.5pt] 
								(0,0) circle (0.1)
								(1.5,0) circle (0.1)     
								(3,0) circle (0.1)
								(4.5,0) circle (0.1)
								(6,0) circle (0.1)
								(7.5,0) circle (0.1)
								(3,1.5) circle (0.1)
								(0.1,0)--(1.4,0)
								(1.6,0)--(2.9,0)
								(3.1,0)--(4.4,0)
								(4.6,0)--(5.9,0)
								(6.1,0)--(7.4,0)
								(3,0.1)--(3,1.4)
								node at (0,-0.4) {1}
								node at (1.5,-0.4) {3}
								node at (3,-0.4) {4}
								node at (4.5,-0.4) {5}
								node at (6,-0.4) {6}
								node at (7.5,-0.4) {7}
								node at (2.6,1.5) {2}
								;
							\end{tikzpicture}
						}
					}
					\\~&~
					\\~&~
					\\~&~
					\\
					\multirow{6.6}{*}{\ \ ${\rm E_8}$} &
					\multirow{5.1}{*}{
						{
							\begin{tikzpicture}[scale=0.8]
								\draw
								[line width=0.5pt] 
								(0,0) circle (0.1)
								(1.5,0) circle (0.1)     
								(3,0) circle (0.1)
								(4.5,0) circle (0.1)
								(6,0) circle (0.1)
								(7.5,0) circle (0.1)
								(9,0) circle (0.1)
								(3,1.5) circle (0.1)
								(0.1,0)--(1.4,0)
								(1.6,0)--(2.9,0)
								(3.1,0)--(4.4,0)
								(4.6,0)--(5.9,0)
								(6.1,0)--(7.4,0)
								(7.6,0)--(8.9,0)
								(3,0.1)--(3,1.4)
								node at (0,-0.4) {1}
								node at (1.5,-0.4) {3}
								node at (3,-0.4) {4}
								node at (4.5,-0.4) {5}
								node at (6,-0.4) {6}
								node at (7.5,-0.4) {7}
								node at (9,-0.4) {8}
								node at (2.6,1.5) {2}
								;
							\end{tikzpicture}
						}
					}
					\\~&~
					\\~&~
					\\~&~
					\\~&~
					\\
				\end{tabular}
			}
		\end{table}
		Suppose $(W,S)$ is a finite Coxeter system.
		\begin{definition}[{\cite[Definition 3.1.1]{GP}}]
			A conjugacy class $C$ of $W$ is called a {\it cuspidal class} if $C\cap W_J=\varnothing$ for all proper sets $J\subset S$.
		\end{definition}
		In other words, a conjugacy class $C$ is a cuspidal class if each expression $w$ of each element $g\in C$ contains all Coxeter generators as letters.\par
		We divide all conjugacy classes $\{C_i:i\in[r]\}$ of $W$ into two parts, non-cuspidal classes and cuspidal classes, denoted by $\{{\it Ncus}_i:i\in[p]\}$ and $\{{\it Cus}_i:i\in[q]\}$ respectively, where $p+q=r$.
		\begin{definition}
			Let signature $\alpha\in\mathbb{N}^n$. The {\it non-cuspidal signature vector} of $\alpha$ is
				\begin{equation*}
					{\it NV}(\alpha)\coloneqq(|E_{\alpha}\cap \phi^{-1}({\it Ncus}_1)|,\ldots,|E_{\alpha}\cap \phi^{-1}({\it Ncus}_p)|).
				\end{equation*}
				The {\it cuspidal signature vector} of $\alpha$ is
				\begin{equation*}
					{\it CV}(\alpha)\coloneqq(|E_{\alpha}\cap \phi^{-1}({\it Cus}_1)|,\ldots,|E_{\alpha}\cap \phi^{-1}({\it Cus}_q)|).
				\end{equation*}
		\end{definition}
		\begin{definition}
				\begin{enumerate}
					\item[{\rm(a)}]
					A sequence $\alpha_1,\ldots,\alpha_p$ of signatures is called a {\it non-cuspidal signature sequence}  (abbreviated as {\rm NSS}) if
					\begin{enumerate}
						\item [1)]
						${\it NV}(\alpha_1),\ldots,{\it NV}(\alpha_p)$ is linearly independent.
						\item [2)]
						$|E_{\alpha_i}\cap \phi^{-1}({\it Cus}_j)|=0$ for each $i\in[p],\ j\in[q]$.
					\end{enumerate}
					Suppose $\alpha_1,\ldots,\alpha_p$ is an {\rm NSS}, the corresponding {\it non-cuspidal signature matrix} (abbreviated as {\rm NSM}) is the matrix
					\begin{equation*}
						\left(
						\begin{array}{c}
							{\it NV}(\alpha_1)\\
							\vdots\\
							{\it NV}(\alpha_p)\\
						\end{array}      	
						\right).        	
					\end{equation*}
					\item[{\rm(b)}]
					A sequence $\alpha_1,\ldots,\alpha_q$ of signatures is called a {\it cuspidal signature sequence} (abbreviated as {\rm CSS}) if ${\it CV}(\alpha_1),\ldots,{\it CV}(\alpha_q)$ is linearly independent. Suppose $\alpha_1,\ldots,\alpha_q$ is a {\rm CSS}, the corresponding {\it cuspidal signature matrix} (abbreviated as {\rm CSM}) is the matrix
					\begin{equation*}
						\left(
						\begin{array}{c}
							{\it CV}(\alpha_1)\\
							\vdots\\
							{\it CV}(\alpha_q)\\
						\end{array}      	
						\right).
					\end{equation*}
				\end{enumerate}
		\end{definition}
		\begin{proposition}\label{cssp}
			\begin{enumerate}
				\item[{\rm(a)}]
				If $W$ has an {\rm NSS} and a {\rm CSS}, then $W$ has an {\rm ISS}.
				\item[{\rm(b)}]
				If $W$ has a lower triangular {\rm NSM} and a lower triangular {\rm CSM}, then $W$ has a lower triangular {\rm ISM}.
			\end{enumerate}
		\end{proposition}
		\begin{proof}
			Suppose $\alpha_1,\ldots, \alpha_p$ is an {\rm NSS} of $W$ and $\beta_1,\ldots, \beta_q$ is a {\rm CSS} of $W$. Arrange all the conjugacy classes as
			\begin{equation*}
				{\it Ncus}_1,\ldots,{\it Ncus}_p,{\it Cus}_1,\ldots,{\it Cus}_q.
			\end{equation*}
			It follows that
			\begin{equation}\label{ucmcm}
				\left(
				\begin{array}{c}
					V_{\alpha_1}\\
					\vdots\\
					V_{\alpha_p}\\
					V_{\beta_1}\\
					\vdots\\
					V_{\beta_q}\\
				\end{array}      	
				\right)
				=
				\begin{pNiceArray}{ccccc|ccccc}[margin]
					&\Block{3-3}{\rm NSM}&&&&&\Block{3-3}{0}&\\
					& & & & & & & & &\\
					& & & & & & & & &\\
					\hline
					&\Block{3-3}{\ddots}&&&&&\Block{3-3}{\rm CSM}&\\
					& & & & & & & & &\\
					& & & & & & & & &\\
				\end{pNiceArray}.      	
			\end{equation}
			It can be verified that $V_{\alpha_1},\ldots,V_{\alpha_p},V_{\beta_1},\ldots,V_{\beta_q}$ is linearly independent and then $\alpha_1,\ldots,\alpha_p,\beta_1,\ldots,\beta_q$ is an {\rm ISS} of $W$. Furthermore, if {\rm NSM} and {\rm CSM} in equation \eqref{ucmcm} are both lower triangular, then  {\rm ISM} of $\alpha_1,\ldots,\alpha_p,\beta_1,\ldots,\beta_q$ is also lower triangular.
		\end{proof}
		Each non-cuspidal class of $W$ arises from a conjugacy class of a proper parabolic subgroup of $W$ and vice versa, whereby we obtain the following theorem.
		\begin{theorem}\label{ucsp}
			If each proper parabolic subgroup of $W$ has an {\rm ISS} with a lower triangular {\rm ISM}, then $W$ has an {\rm NSS} with a lower triangular {\rm NSM}.
		\end{theorem}
		
		\begin{proof}
			Let {\it U} denote the set of all non-cuspidal classes of $W$. Let $W_{J_1},\ldots,W_{J_m}$ be all proper parabolic subgroups of $W$. 
			\par
			For each proper parabolic subgroups $W_{J_k}$, by assumption, there is an order of all conjugacy classes of $W_{J_k}$: $C_1^{J_k},\ldots,C_{r_k}^{J_k}$ and a sequence $\alpha_{k,1},\ldots,\alpha_{k,r_k} \\ \in\mathbb{N}^{r_k}$ such that
			\begin{equation}\label{56eq1}
			\begin{aligned}
				&E_{\alpha_{k,i}}\cap\phi^{-1}(C_i^{J_k})\neq\varnothing\quad\text{for}\ i\in[r_k],
				\\
				&E_{\alpha_{k,i}}\cap\phi^{-1}(C_j^{J_k})=\varnothing\quad\text{for}\ i,j\in[r_k]\ \text{and}\ j>i.
			\end{aligned}
		    \end{equation}
			Let
			\begin{equation*}
				\psi:  C_i^{J_k}\mapsto \{gxg^{-1}:x\in C_i^{J_k},g\in W\}.
			\end{equation*}
			$\psi(C_i^{J_k})$ is the conjugacy class of $W$ that contains $C_i^{J_k}$. It is easy to see that each $\psi(C_i^{J_k})$ is an non-cuspidal class of $W$. Furthermore, since each non-cuspidal class of $W$ contains at least one word that belongs to some proper parabolic subgroup, we have
			\begin{equation*}
				\bigcup_{k=1}^m\bigcup_{i=1}^{r_k}\psi(C_i^{J_k})={\it U}.
			\end{equation*}
			Suppose $J_k=\{s_{k_1},\ldots,s_{k_d}\}$ where $k_1<\cdots<k_d$. Given $\alpha_{k,i}=(a_1,\ldots,a_d)$, define $\alpha'_{k,i}=(a'_1,\ldots,a'_n)$ where
			\begin{equation*}
				a'_x=
				\begin{cases}
					a_i &{\rm if}\ x=k_i\ {\rm for\ some}\ i\in[d],\\
					0 &{\rm else}.
				\end{cases}
			\end{equation*}
			for each $x\in[n]$. In another word, $\alpha'_{k,i}$ is the signature on $W$ induced by $\alpha_{k,i}$. Let ${\it U}=\{{\it Ncus}_1,{\it Ncus}_2,\ldots,{\it Ncus}_p\}$ where
			\begin{equation}\label{56eq2}
				\min_{\leq_{\rm lex}}\{(k,i):\psi(C_i^{J_k})={\it Ncus}_1\}\leq_{\rm lex}\cdots\leq_{\rm lex}\min_{\leq_{\rm lex}}\{(k,i):\psi(C_i^{J_k})={\it Ncus}_p\}.
			\end{equation}
			Let
			\begin{equation}\label{56eq3}
				\beta_j=\alpha'_{\theta},\ \theta=\min_{\leq_{\rm lex}}\{(k,i):\psi(C_i^{J_k})={\it Ncus}_j\}
			\end{equation}
			\par
			The below shows that $\beta_1,\ldots,\beta_p$ is an {\rm NSS} with a lower triangular {\rm NSM} of $W$. For any given $\beta_j$, let $\beta_j=\alpha'_{k,i}$. First, since $\phi(E_{\beta_j})\subset W_{J_k}$ and $W_{J_k}\cap{\it Cus}_l=\varnothing$ for each $l\in[q]$, we have $E_{\beta_j}\cap\phi^{-1}({\it Cus}_l)=\varnothing$ for each $l\in[q]$. Next, by \eqref{56eq1} we have $E_{\beta_j}\cap \phi^{-1}(C_i^{J_k})\neq\varnothing$, by \eqref{56eq3} we have $\psi(C_i^{J_k})={\it Ncus}_j$, thus $E_{\beta_j}\cap\phi^{-1}({\it Ncus}_j)\neq\varnothing$. Suppose $E_{\beta_j}\cap\phi^{-1}({\it Ncus}_h)\neq\varnothing$. By \eqref{56eq1}, we obtain ${\it Ncus}_h=\psi(C_{i'}^{J_k})$ for some $i'\leq i$. Let $\beta_h=\alpha'_{\tilde{k},\tilde{i}}$, then $(\tilde{k},\tilde{i})\leq_{\rm lex}(k,i')\leq_{\rm lex}(k,i)$ by \eqref{56eq3}, and then $h\leq j$ by \eqref{56eq2}. Thus, for each $h>j$, $E_{\beta_j}\cap\phi^{-1}({\it Ncus}_h)=\varnothing$. Therefore $\beta_1,\ldots,\beta_p$ is an {\rm NSS} with a lower triangular {\rm NSM} of $W$.
		\end{proof}
		\begin{theorem}\label{cmp}
			Each Coxeter group of exceptional types has a {\rm CSS} with a lower triangular {\rm CSM}.
		\end{theorem}
		We will prove this theorem in the next section. The proof is largely based on the results obtained by running the algorithm we developed in MATLAB. Once the theorem holds, we can prove the main result of this section.
		\begin{theorem}\label{expp}
			Each Coxeter group of exceptional types has an {\rm ISS} with a lower triangular {\rm ISM}.
		\end{theorem}
		\begin{proof}
			By Proposition \refeq{cssp} and Theorem \refeq{cmp}, it is sufficient to show that each Coxeter group $W(M)$ of exceptional types has an {\rm NSS} with a lower triangular {\rm NSM}. Furthermore, by Theorem \refeq{ucsp}, it is sufficient to show that each proper parabolic subgroup of $W(M)$ has an {\rm ISS} with a lower triangular {\rm ISM}.\par 
			Consider all Coxeter groups of exceptional types as the order ${\rm H_3},{\rm H_4},{\rm F_4},\\{\rm E_6},{\rm E_7},{\rm E_8}$. By observing Table \refeq{table1} and Table \refeq{table2}, given an exceptional type $M$,\\ there are only three possible cases for the proper parabolic subgroups of $W(M)$:
			\begin{enumerate}
				\item [{\rm a)}]
				classical types,
				\item [{\rm b)}]
				exceptional types which are before it,
				\item [{\rm c)}]
				direct products of some types in the first two cases.
			\end{enumerate}
			By Theorem \refeq{clsp1}, \refeq{clsp2} and Corollary \refeq{dpc}, for each exceptional type $M$, all that is needed is to show that every previous exceptional types have an {\rm ISS} with a lower triangular {\rm ISM}. Hence ${\rm H_3}$ is proven firstly, and then ${\rm H_4,F_4,E_6,E_7}$ and ${\rm E_8}$ are proven in turn, so the theorem holds.
		\end{proof}
		
		\section{The proof of Theorem \refeq{cmp} and algorithms}
		\hspace{1.5em}Suppose $(W,S)$ is a finite Coxeter system, here we define some notations. For $g\in W$, $w\in M(S)$,
		\begin{align*}
			l(g)&\coloneqq{\rm min}\{|u|:u\in M(S),\phi(u)=g\},\\
			l(w)&\coloneqq{\rm min}\{|u|:u\in M(S),\phi(u)=\phi(w)\}.
		\end{align*}
		For a conjugacy class $C$ of $W$,
		\begin{equation*}
			l(C)\coloneqq{\rm min}\{l(g):g\in C\}.
		\end{equation*}
		Recall from Section \refeq{sec2} that $\rho_{\it ref}$ denotes the reflection representation of $W$. Characteristic polynomial of $g\in W$ is
		\begin{equation*}
			p_g(\lambda)\coloneqq\det[\lambda I-\rho_{\it ref}(g)].
		\end{equation*}
		Characteristic polynomial of a conjugacy class $C$ is 
		\begin{equation*}
			p_C(\lambda)\coloneqq p_g(\lambda)
		\end{equation*}
		for an arbitrary element $g\in C$.
		\begin{lemma}\label{cuslemma}
			Let ${\it Cus}_1,\ldots,{\it Cus}_q$ be all cuspidal classes of $W$ with $l({\it Cus}_1)\leq\cdots\leq l({\it Cus}_q)$.
			\begin{enumerate}
				\item[{\rm(a)}]
				If $l({\it Cus}_i)\neq l({\it Cus}_j)$ for each $i\neq j$, then $W$ has a {\rm CSS} with a lower triangular {\rm CSM}.
				\item[{\rm(b)}]
				If
				\begin{multline*}
					l({\it Cus}_1)<\cdots<l({\it Cus}_{k_1})=l({\it Cus}_{k_1+1})<\cdots\\<l({\it Cus}_{k_h})=l({\it Cus}_{k_h+1})<\cdots<l({\it Cus}_q)
				\end{multline*}
				and for each $j\in[h]$, there exists a word $w_{k_j}\in\phi^{-1}(\it Cus_{k_j})$ such that $|w_{k_j}|=l({\it Cus}_{k_j})$ and $|E_{{\rm sig}(w_{k_j})}\cap \phi^{-1}({\it Cus}_{k_j+1})|=0$, then $W$ has a {\rm CSS} with a lower triangular {\rm CSM}.
			\end{enumerate}
		\end{lemma}
		\begin{proof}
			\begin{enumerate}
				\item[{\rm(a)}]
				For each $i\in[q]$, take a word $w_i\in \phi^{-1}({\it Cus}_i)$ such that $|w_i|=l({\it Cus}_i)$. Then $|E_{{\rm sig}(w_i)}\cap \phi^{-1}({\it Cus}_j)|=0$ for all $j>i$ since $|w_i|=l({\it Cus}_i)<l({\it Cus}_j)$. Also, clearly we have $|E_{{\rm sig}(w_i)}\cap \phi^{-1}({\it Cus}_i)|\neq0$. Hence ${\rm sig}(w_1),\ldots,{\rm sig}(w_q)$ is a {\rm CSS} with a lower triangular {\rm CSM}.
				\item[{\rm(b)}]
				For each $i\in[q]$, if $i=k_j$ for some $j\in[h]$, take $w_i$ such that $|w_i|=l({\it Cus}_{i})$ and $|E_{{\rm sig}(w_i)}\cap \phi^{-1}({\it Cus}_{i+1})|=0$ by assumption; or else, take $w_i\in \phi^{-1}({\it Cus}_i)$ such that $|w_i|=l({\it Cus}_i)$. Since $|E_{{\rm sig}(w_i)}\cap \phi^{-1}({\it Cus}_i)|\neq0$ and $|E_{{\rm sig}(w_i)}\cap \phi^{-1}({\it Cus}_j)|=0$ for all $j>i$, it follows that ${\rm sig}(w_1),\ldots,{\rm sig}(w_q)$ is a {\rm CSS} with a lower triangular {\rm CSM}.
			\end{enumerate}
		\end{proof}
		M. Geck and G. Pfeiffer listed all cuspidal classes of each exceptional type in \cite[Appendix B]{GP}. We summarize the information and obtain the following proposition.
		\begin{proposition}[{\cite[Appendix B]{GP}}]\label{cusprop}
			Let ${\it Cus}_i^{\rm M}$ denote the $i$th cuspidal class of $W(M)$ in \cite[Appendix B]{GP}.
			\begin{enumerate}
				\item[{\rm(a)}]
				For $W=W({\rm H_3})$, $l({\it Cus}_1^{\rm H_3})<\cdots<l({\it Cus}_4^{\rm H_3})$.
				\item[{\rm(b)}]
				For $W=W({\rm H_4})$,
				\begin{equation*}
					l({\it Cus}_1^{\rm H_4})<\cdots<l({\it Cus}_7^{\rm H_4})=l({\it Cus}_8^{\rm H_4})<\cdots<l({\it Cus}_{20}^{\rm H_4}).
				\end{equation*}
				\item[{\rm(c)}]
				For $W=W({\rm F_4})$,
				\begin{equation*}
					l({\it Cus}_1^{\rm F_4})<\cdots<l({\it Cus}_4^{\rm F_4})=l({\it Cus}_5^{\rm F_4})<\cdots<l({\it Cus}_9^{\rm F_4}).
				\end{equation*}
				\item[{\rm(d)}]
				For $W=W({\rm E_6})$, $l({\it Cus}_1^{\rm E_6})<\cdots<l({\it Cus}_5^{\rm E_6})$.
				\item[{\rm(e)}]
				For $W=W({\rm E_7})$, $l({\it Cus}_1^{\rm E_7})<\cdots<l({\it Cus}_{12}^{\rm E_7})$.
				\item[{\rm(f)}]
				For $W=W({\rm E_8})$,
				\begin{gather*}
					l({\it Cus}_1^{\rm E_8})<\cdots<l({\it Cus}_5^{\rm E_8})=l({\it Cus}_6^{\rm E_8})<\cdots<l({\it Cus}_9^{\rm E_8})=l({\it Cus}_{10}^{\rm E_8})<\\
					l({\it Cus}_{11}^{\rm E_8})=l({\it Cus}_{12}^{\rm E_8})<l({\it Cus}_{13}^{\rm E_8})=l({\it Cus}_{14}^{\rm E_8})<\cdots\\
					<l({\it Cus}_{21}^{\rm E_8})=l({\it Cus}_{22}^{\rm E_8})<l({\it Cus}_{23}^{\rm E_8})=l({\it Cus}_{24}^{\rm E_8})<\cdots<l({\it Cus}_{30}^{\rm E_8}).
				\end{gather*}
			\end{enumerate}
		\end{proposition}
		
		\begin{lemma}[{\cite[Lemma 3.1.10, Exercise 3.16]{GP}}]\label{cuspolylemma}
			Let $g\in W$, the conjugacy class of $g$ in $W$ is a cuspidal class if and only if $p_g(1)\neq 0$.
		\end{lemma}
		\begin{proposition}\label{pgprop}
				For $W=W({\rm H_3}),W({\rm H_4}),W({\rm E_6}),W({\rm E_7}),\text{or}\ W({\rm E_8})$, let $C_0$ be a cuspidal class of $W$. Let $g\in W$, then $g\in C_0$ if and only if $p_g(\lambda)=p_{C_0}(\lambda)$.
		\end{proposition}
		\begin{proof}
			If $g\in C_0$, then $p_g(\lambda)=p_{C_0}(\lambda)$. Conversely, if $p_g(\lambda)=p_{C_0}(\lambda)$, it follows that $p_g(1)\neq 0$ so the conjugacy class of $g$ in $W$ is a cuspidal class by Lemma \refeq{cuspolylemma}. It is known that the characteristic polynomials of any two different cuspidal classes in $W$ are different by \cite[Appendix B]{GP}. Hence $g$ is in $C_0$.
		\end{proof}
		\begin{remark}
			For $W=W({\rm F_4})$, the result in Proposition \refeq{pgprop} is false since $p_{{\it Cus}_4^{\rm F_4}}(\lambda)=p_{{\it Cus}_5^{\rm F_4}}(\lambda)=(\lambda+1)(\lambda^3+1)$.
		\end{remark}
		\begin{remark}
		By Lemma \refeq{cuslemma} and Proposition \refeq{cusprop}, it remains for us to show that for each two cuspidal classes ${\it Cus}_i^M$ and ${\it Cus}_j^M$ that $M={\rm H_4,F_4}$ or ${\rm E_8}$ and $l({\it Cus}_i^M)=l({\it Cus}_j^M)$, there exists a word $w\in\phi^{-1}({\it Cus}_i^M)$ such that $|w|=l({\it Cus}_i^M)$ and $E_{{\rm sig}(w)}\cap \phi^{-1}({\it Cus}_j^M)=\varnothing$. It is a natural idea to design an algorithm to check whether $E_{{\rm sig}(w)}\cap \phi^{-1}({\it Cus}_j^M)=\varnothing$ for given ${\it Cus}_i^M,{\it Cus}_j^M$ and word $w$ such that $\phi(w)\in {\it Cus}_i^M$ and $|w|=l({\it Cus}_i^M)$. Specifically, such an algorithm should realize the following functions:
		\begin{enumerate}
			\item [{\rm a)}]
			Traverse $E_{{\rm sig}(w)}$.
			\item [{\rm b)}]
			For each $u\in E_{{\rm sig}(w)}$, check whether $\phi(u)\in {\it Cus}_j^M$.
		\end{enumerate}
		\par
		However, such $E_{{\rm sig}(w)}$ can be very large, especially in $W({\rm E_8})$. To be precise, suppose ${\rm sig}(w)=(a_1,\ldots,a_n)$, then $|E_{{\rm sig}(w)}|=|{\rm sig}(w)|!/(a_1!\cdots a_n!)$. Thus the traversal usually cannot be completed within an acceptable time. So we hope to find a small enough subset $H\subseteq E_{{\rm sig}(w)}$ such that
		\begin{equation*}
			E_{{\rm sig}(w)}\cap\phi^{-1}({\it Cus}_j^M)=\varnothing\quad\text{if and only if}\quad H\cap\phi^{-1}({\it Cus}_j^M)=\varnothing
		\end{equation*}
		so that we just need to traverse $H$ instead of $E_{{\rm sig}(w)}$.
    	\end{remark}
		\begin{definition}
			Suppose $(W,S)$ is a Coxeter system, let $w=t_1t_2\cdots t_n\in M(S)$, $t_1,\ldots,t_n\in S$. The {\it cyclic class} of $w$ is a subset of $M(S)$ defined as
			\begin{equation*}
				\{t_1\cdots t_n,t_2\cdots t_nt_1,\ldots,t_nt_1\cdots t_{n-1}\}.
			\end{equation*}
		\end{definition}
		For $W=W({\rm H_4}),W({\rm F_4})$, let signature $\alpha=(a_1,a_2,a_3,a_4)$ with $a_1\leq a_2$ and $a_4\leq a_3$, and let $\alpha_{23}=(0,a_2,a_3,0)$. Let $u\in E_{\alpha_{23}}$, choose $a_1$ letters from all $a_2$ letters of $s_2$ in $u$ and insert $s_1$ to the left of each chosen letters. Let $I_1^{\alpha}(u)$ be the set consisting of all the words obtained by this way. For example, suppose $u=s_2s_3s_2s_2$ and $a_1=2$, then
		\begin{equation*}
			I_1^{\alpha}(u)=\{s_1s_2s_3s_1s_2s_2,s_1s_2s_3s_2s_1s_2,s_2s_3s_1s_2s_1s_2\}.
		\end{equation*}
		Similarly, let $u\in I_1^{\alpha}(E_{\alpha_{23}})$, choose $a_4$ letters from all $a_3$ letters of $s_3$ in $u$ and insert $s_4$ to the left of each chosen letter and obtain a set $I_4^{\alpha}(u)$ consisting of all the words derived by this way. We write $I_4^{\alpha}(I_1^{\alpha}(\cdot))$ as $I_4^{\alpha}I_1^{\alpha}(\cdot)$ for short.
		\par 
		All the cyclic classes in $E_{\alpha_{23}}$ form a partition of $E_{\alpha_{23}}$. We select an arbitrary element from each cyclic class in $E_{\alpha_{23}}$ and obtain a set $E_{\alpha_{23}}^{\rm cyc}$ consisting of these elements. 
		\begin{proposition}\label{algprop}
			Let $(W(M),S)$ be a Coxeter system and $M={\rm H_4}\ \text{or}\ {\rm F_4}$. Suppose two cuspidal classes ${\it Cus}_i^M$ and ${\it Cus}_j^M$ satisfy $l({\it Cus}_i^M)=l({\it Cus}_j^M)$, and a word $w\in\phi^{-1}({\it Cus}_i^M)$ satisfies $|w|=l({\it Cus}_i^M)$ and let $\alpha={\rm sig}(w)$. Then $E_{\alpha}\cap\phi^{-1}({\it Cus}_j^M) =\varnothing$ if and only if $I_4^{\alpha}I_1^{\alpha}(E_{{\alpha}_{23}}^{\rm cyc})\cap\phi^{-1}({\it Cus}_j^M)=\varnothing$.
		\end{proposition}
		\begin{proof}
			Let ${\rm sig}(w)=(a_1,a_2,a_3,a_4)$. It follows that $a_1\leq a_2$ and $a_4\leq a_3$, since $|w|=l({\it Cus}_i^M)$. It is obvious that $I_4^{\alpha}I_1^{\alpha}(E_{{\alpha}_{23}}^{\rm cyc})\subseteq E_{\alpha}$, hence $E_{\alpha}\cap\phi^{-1}({\it Cus}_j^M)=\varnothing$ implies $I_4^{\alpha}I_1^{\alpha}(E_{{\alpha}_{23}}^{\rm cyc})\cap\phi^{-1}({\it Cus}_j^M)=\varnothing$.
			
			Conversely, suppose $I_4^{\alpha}I_1^{\alpha}(E_{{\alpha}_{23}}^{\rm cyc})\cap\phi^{-1}({\it Cus}_j^M)=\varnothing$. Since each element in $I_4^{\alpha}I_1^{\alpha}(E_{{\alpha}_{23}})$ is conjugate to an element in $I_4^{\alpha}I_1^{\alpha}(E_{{\alpha}_{23}}^{\rm cyc})$, we have $I_4^{\alpha}I_1^{\alpha}(E_{{\alpha}_{23}})\cap\phi^{-1}({\it Cus}_j^M)=\varnothing$. Let $u\in E_{\alpha}$ and $C_0$ is the conjugacy class of $\phi(u)$. If $|u|>l(C_0)$, then $u\notin\phi^{-1}({\it Cus}_j^M)$ since $|u|=l({\it Cus}_j^M)$. If $|u|=l(C_0)$, we perform the following operations on $u$. First, for each letter $s_1$ in $u$, swap it with letters to its right until a letter non-commuting with $s_1$ is encountered. If this $s_1$ reaches the rightmost position, then move it to the leftmost position and repeat the above step. Since $|u|=l(C_0)$, no $s_1$ will encounter another $s_1$. Thus every $s_1$ will eventually encounter an $s_2$ to its right through this procedure. Next, perform a similar procedure for each $s_4$ in the newly obtained word, then every $s_4$ will eventually encounter an $s_3$ to its right. Therefore, we finally obtain a new word $u'\in I_4^{\alpha}I_1^{\alpha}(E_{{\alpha}_{23}})$ and $\phi(u')$ is conjugate to $\phi(u)$. Since $I_4^{\alpha}I_1^{\alpha}(E_{{\alpha}_{23}})\cap\phi^{-1}({\it Cus}_j^M)=\varnothing$, we have $u'\notin\phi^{-1}({\it Cus}_j^M)$ and then $u\notin\phi^{-1}({\it Cus}_j^M)$. Hence $E_{\alpha}\cap\phi^{-1}({\it Cus}_j^M)=\varnothing$.
		\end{proof}
		For $W=W({\rm E_8})$, let signature $\alpha=(a_1,a_2,a_3,a_4,a_5,a_6,a_7,a_8)$ where $a_1\leq a_3,a_1\leq a_4,a_3-a_1\leq a_4,a_2\leq a_4,a_5\leq a_4+a_6,a_8\leq a_7,a_8\leq a_6$ and $a_7-a_8\leq a_6$ (otherwise $|u|>l(C)$ for each $u\in E_{\alpha}$, where $C$ is the conjugacy class of $\phi(w)$). Similar to $I_1^{\alpha}$ and $I_4^{\alpha}$ in types ${\rm H_4}$ and ${\rm F_4}$, we define the following operations. Let $\alpha_{46}=(0,0,0,a_4,0,a_6,0,0)$.
		\par
		\begin{enumerate}[a)]
		\item
		Let $u\in E_{\alpha_{46}}$, choose $a_3-a_1$ letters from all $a_4$ letters of $s_4$ in $u$ and insert $s_3$ to the left of each chosen letter and obtain a set $I_3^{\alpha}(u)$ consisting of all the words derived by this way.
		\item
		Let $u\in I_3^{\alpha}(E_{\alpha_{46}})$, choose $a_1$ letters from all $a_4$ letters of $s_4$ in $u$ and insert $s_1s_3$ to the left of each chosen letter and obtain a set $I_{13}^{\alpha}(u)$ consisting of all the words derived by this way.
		\item
		Let $u\in I_{13}^{\alpha}I_3^{\alpha}(E_{\alpha_{46}})$, choose $a_2$ letters from all $a_4$ letters of $s_4$ in $u$ and insert $s_2$ to the left of each chosen letter and obtain a set $I_2^{\alpha}(u)$ consisting of all the words derived by this way.
		\item
		Let $u\in I_2^{\alpha}I_{13}^{\alpha}I_3^{\alpha}(E_{\alpha_{46}})$, choose $a_5$ letters from all $a_4+a_6$ letters of $s_4,s_6$ in $u$ and insert $s_5$ to the left of each chosen letter and obtain a set $I_5^{\alpha}(u)$ consisting of all the words derived by this way.
		\item
		Let $u\in I_5^{\alpha}I_2^{\alpha}I_{13}^{\alpha}I_3^{\alpha}(E_{\alpha_{46}})$, choose $a_7-a_8$ letters from all $a_6$ letters of $s_6$ in $u$ and insert $s_7$ to the left of each chosen letter and obtain a set $I_7^{\alpha}(u)$ consisting of all the words derived by this way.
		\item
		Let $u\in I_7^{\alpha}I_5^{\alpha}I_2^{\alpha}I_{13}^{\alpha}I_3^{\alpha}(E_{\alpha_{46}})$, choose $a_8$ letters from all $a_6$ letters of $s_6$ in $u$ and insert $s_8s_7$ to the left of each chosen letter and obtain a set $I_{87}^{\alpha}(u)$ consisting of all the words derived by this way.
    	\end{enumerate}
		\par
		By selecting an arbitrary element from each cyclic class in $E_{\alpha_{46}}$, we obtain a set $E_{\alpha_{46}}^{\rm cyc}$ consisting of these elements. By a proof similar to that of Proposition \refeq{algprop}, the following proposition can be established.
		\begin{proposition}
			For Coxeter system $(W({\rm E_8}),S)$, suppose two cuspidal classes ${\it Cus}_i^{\rm E_8}$ and ${\it Cus}_j^{\rm E_8}$ satisfy $l({\it Cus}_i^{\rm E_8})=l({\it Cus}_j^{\rm E_8})$, word $w\in\phi^{-1}({\it Cus}_i^{\rm E_8})$ satisfies $|w|=l({\it Cus}_i^{\rm E_8})$ and let $\alpha={\rm sig}(w)$. Then $E_{\alpha}\cap\phi^{-1}({\it Cus}_j^{\rm E_8})=\varnothing$ if and only if $I_{87}^{\alpha}I_7^{\alpha}I_5^{\alpha}I_2^{\alpha}I_{13}^{\alpha}I_3^{\alpha}(E_{{\alpha}_{46}}^{\rm cyc})\cap\phi^{-1}({\it Cus}_j^{\rm E_8})=\varnothing$.
		\end{proposition}
		In the rest of this paper, we may write word $s_is_js_k\ldots$ as $ijk\ldots$ for short.
		\begin{proposition}\label{f4prop}
			For $W=W({\rm F_4})$, there exists $w\in\phi^{-1}({\it Cus}_5^{\rm F_4})$ such that $|w|=l({\it Cus}_5^{\rm F_4})$ and $E_{{\rm sig}(w)}\cap\phi^{-1}({\it Cus}_4^{\rm F_4})=\varnothing$.
		\end{proposition}
		\begin{proof}
			Let $w=1213213234$ which is a representative of ${\it Cus}_5^{\rm F_4}$ listed in \cite[Appendix B]{GP}. Let $\alpha={\rm sig}(w)$, then $\alpha=(3,3,3,1)$ and ${\alpha}_{23}=(0,3,3,0)$. Select an arbitrary element from each cyclic class in $E_{{\alpha}_{23}}$ and obtain
			\begin{equation*}
				E_{\alpha_{23}}^{\rm cyc}=\{222333,223233,232233,232323\},
			\end{equation*}
			and then
			\begin{equation}
			\begin{aligned}\label{f4set}
				I_4^{\alpha}I_1^{\alpha}(E_{{\alpha}_{23}}^{\rm cyc})=\{&1212124333,1212123433,1212123343,1212431233,\\&1212312433,1212312343,1243121233,1231212433,\\&1231212343,1243123123,1231243123,1231231243\}.
			\end{aligned}
		    \end{equation}
			Note that $\phi(1212)=\phi(21)$ and the last three words in \eqref{f4set} are all in $\phi^{-1}({\it Cus}_5^{\rm F_4})$, so $I_4^{\alpha}I_1^{\alpha}(E_{{\alpha}_{23}}^{\rm cyc})\cap\phi^{-1}({\it Cus}_4^{\rm F_4})=\varnothing$. By Proposition \refeq{algprop}, we have $E_{\alpha}\cap\phi^{-1}({\it Cus}_4^{\rm F_4})=\varnothing$.
		\end{proof}
		For $W=W({\rm H_4})$ or $W({\rm E_8})$, similar conclusions are not as obvious as for $W({\rm F_4})$, so we perform calculations in MATLAB to solve them.
		\begin{alg}\label{algA}
			Given a Coxeter system $(W({\rm M}),S)$ where $M={\rm H_4}$ or ${\rm E_8}$,  two cuspidal classes ${\it Cus}_i^M$ and ${\it Cus}_j^M$ such that $l({\it Cus}_i^M)=l({\it Cus}_j^M)$, a word $w\in \phi^{-1}({\it Cus}_i^M)$ such that $|w|=l({\it Cus}_i^M)$ and let $\alpha={\rm sig}(w)$, check whether $E_{\alpha}\cap\phi^{-1}({\it Cus}_j^M)=\varnothing$.
			\begin{enumerate}
				\item [\textbf{A1.}]
				{\it Flag} $\leftarrow 0$. Calculate $p_{{\it Cus}_j^{\rm H_4}}(\lambda)$ (resp. $p_{{\it Cus}_j^{\rm E_8}}(\lambda)$).
				\item [\textbf{A2.}]
				Traverse $I_4^{\alpha}I_1^{\alpha}(E_{{\alpha}_{23}}^{\rm cyc})$ (resp. $I_{87}^{\alpha}I_7^{\alpha}I_5^{\alpha}I_2^{\alpha}I_{13}^{\alpha}I_3^{\alpha}(E_{{\alpha}_{46}}^{\rm cyc})$). For each $v\in I_4^{\alpha}I_1^{\alpha} \\ (E_{{\alpha}_{23}}^{\rm cyc})$ (resp. $I_{87}^{\alpha}I_7^{\alpha}I_5^{\alpha}I_2^{\alpha}I_{13}^{\alpha}I_3^{\alpha}(E_{{\alpha}_{46}}^{\rm cyc})$), calculate $p_{\phi(v)}(\lambda)$.
				\begin{enumerate}
					\item [\textbf{A2.1.}]
					If $p_{\phi(v)}(\lambda)=p_{{\it Cus}_j^{\rm H_3}}(\lambda)$ (resp. $p_{\phi(v)}(\lambda)=p_{{\it Cus}_j^{\rm E_8}}(\lambda)$), then {\it Flag} $\leftarrow 1$ and the traversal terminates.
				\end{enumerate}
				\item [\textbf{A3.}]
				Output the value of {\it Flag}.
			\end{enumerate}
		\end{alg}
		\begin{remark}
			By Proposition \refeq{pgprop} and \refeq{algprop}, $E_{{\rm sig}(w)}\cap\phi^{-1}({\it Cus}_j^M)=\varnothing$ if and only if running algorithm \refeq{algA} yields ${\it Flag}=0$.
		\end{remark}
		\begin{remark}\label{h4e8rmk}
			For $W({\rm H_4})$, let
			\begin{flalign*}
			&w=1212132121321234\ {\rm for}\ {\it Cus}_7^{\rm H_4},{\it Cus}_8^{\rm H_4}.&
		    \end{flalign*}
		    For $W({\rm E_8})$, let
		    \begin{flalign*}
		    &w=1231423454657658\ {\rm for}\ {\it Cus}_5^{\rm E_8},{\it Cus}_6^{\rm E_8};&\\
		    &w=1234234542345654765876\ {\rm for}\ {\it Cus}_9^{\rm E_8},{\it Cus}_{10}^{\rm E_8};&\\
		    &w=123142314542345654765876\ {\rm for}\ {\it Cus}_{11}^{\rm E_8},{\it Cus}_{12}^{\rm E_8};&\\
		    &w=12314231454231456542345678\ {\rm for}\ {\it Cus}_{13}^{\rm E_8},{\it Cus}_{14}^{\rm E_8};&\\
		    &w=12314231545231436542314354265431765423456878\ {\rm for}\ {\it Cus}_{22}^{\rm E_8},{\it Cus}_{21}^{\rm E_8};&\\
		    &w=1231423154523165456237654567238765456782345678\ {\rm for}\ {\it Cus}_{23}^{\rm E_8},{\it Cus}_{24}^{\rm E_8}.&
		    \end{flalign*}
		    By running Algorithm \refeq{algA}, we get ${\it Flag}=0$ in each case. Hence these $w$ satisfy $w\in\phi^{-1}({\it Cus}_i^M)$, $|w|=l({\it Cus}_i^M)$ and $E_{{\rm sig}(w)}\cap \phi^{-1}({\it Cus}_j^M)=\varnothing$.
		\end{remark}
		\begin{proof}[Proof of Theorem \refeq{cmp}]
			By Lemma \refeq{cuslemma} and Proposition \refeq{cusprop}, it suffices to show that for each two cuspidal classes ${\it Cus}_i^M$ and ${\it Cus}_j^M$ that $M={\rm H_4},{\rm F_4},\text{or}\ {\rm E_8}$ and $l({\it Cus}_i^M)=l({\it Cus}_j^M)$, there exists a word $w\in\phi^{-1}({\it Cus}_i^M)$ such that $|w|=l({\it Cus}_i^M)$ and $E_{{\rm sig}(w)}\cap \phi^{-1}({\it Cus}_j^M)=\varnothing$. By Proposition \refeq{f4prop} and Remark \refeq{h4e8rmk}, such $w$ exists for each case, so the theorem holds. 
		\end{proof}
		
		\section{The proof of Theorem \refeq{MT} and examples}
		\begin{proof}[Proof of Theorem \refeq{MT}]
	    	By Theorem \refeq{clsp1}, \refeq{clsp2} and \refeq{expp}, any finite irreducible Coxeter group has an {\rm ISS} with a lower triangular {\rm ISM}. Thus any finite Coxeter group has an {\rm ISS} with a lower triangular {\rm ISM} by Corollary \refeq{dpc}.
	    	\par
	    	Let $W$ be a finite Coxeter group, $\rho_1$ and $\rho_2$ are two finite dimensional complex linear representations of $W$. It is obvious that $\rho_1\cong\rho_2$ implies $d(S,\rho_1)=d(S,\rho_2)$. Conversely, by Proposition \refeq{keyprop}, $d(S,\rho_1)=d(S,\rho_2)$ implies $\rho_1\cong\rho_2$.
    	\end{proof}
		Suppose $(W,S)$ is a finite Coxeter system. Recall that ${\mathcal R}(W)$ is the set of all finite dimensional complex linear representations of $W$. Let
		\begin{equation*}
			\mathscr{D}(W,S)\coloneqq\{d(S,\rho):\rho\in{\mathcal R}(W)\}.
		\end{equation*}
		$\mathscr{D}(W,S)$ is a semigroup with respect to polynomials multiplication. Let $\widehat{\mathcal R}(W)$ denote the semigroup of equivalence classes of ${\mathcal R}(W)$, with the semigroup operation defined by the direct sum of representations. By Theorem \refeq{MT} and the equation
		\begin{equation*}
			d(S,\bigoplus_{i=1}^k\rho_i)=\prod_{i=1}^kd(S,\rho_i),
		\end{equation*}
		the following corollary holds.
		\begin{corollary}\label{MC}
			Let $(W,S)$ be a finite Coxeter system. Let $\rho_1,\ldots,\rho_r$ be a complete list of irreducible representations of $W$.
			\begin{enumerate}
				\item[{\rm(a)}]
				For $\rho\in\mathcal{R}(W)$, $\rho\cong\bigoplus_{i=1}^rk_i\rho_i$ if and only if $d(S,\rho)=\prod_{i=1}^{r}d^{k_i}(S,\rho_i)$.
				\item[{\rm(b)}]
			    There is a semigroup isomorphism
			    \begin{equation*}
			    	\widehat{\mathcal R}(W)\cong\mathscr{D}(W,S).
			    \end{equation*}
			    In particular, $\mathscr{D}(W,S)=\langle d(S,\rho_1),\ldots,d(S,\rho_r)\rangle$ is a free commutative semigroup.
		    \end{enumerate}
		\end{corollary}
		\begin{remark}
			Let
			\begin{equation*}
				\tilde{d}(S,\rho)\coloneqq\det[I+x_1\rho(s_1)+\cdots+x_n\rho(s_n)]
			\end{equation*}
			and
			\begin{equation*}
				\widetilde{\mathscr{D}}(W,S)\coloneqq\{\tilde{d}(S,\rho):\rho\in{\mathcal R}(W)\}.
			\end{equation*}
			Replacing $d$ and $\mathscr{D}$ by $\tilde{d}$ and $\widetilde{\mathscr{D}}$, Theorem \refeq{MT} and Corollary \refeq{MC} still hold.
		\end{remark}
		\begin{example}
			For $W=W({\rm H_3})$, there are 10 conjugacy classes
			\begin{align*}
				&C(\varepsilon), C(3), C(23), C(13), C(12), C(123), C(2323),\\ &C(12323), C(123212323), C(123213232132323).
			\end{align*}
			where $C(w)\coloneqq\{g\phi(w)g^{-1}:g\in W\}$. Let
			\begin{align*}
				&\alpha_1=(0,0,0), \alpha_2=(0,0,1), \alpha_3=(0,1,1), \alpha_4=(1,0,1), \alpha_5=(1,1,0),\\ &\alpha_6=(1,1,1), \alpha_7=(0,2,2), \alpha_8=(1,2,2), \alpha_9=(2,4,3), \alpha_{10}=(3,6,6).
			\end{align*}
			By performing calculations in MATLAB, it follows that
			\begin{equation*}
				\begin{pmatrix}
					V_{\alpha_1}\\
					V_{\alpha_2}\\
					\vdots\\
					V_{\alpha_{10}}\\
				\end{pmatrix}      	
				=
				\begin{pmatrix}
					1&0&0&0&0&0&0&0&0&0\\
					0&1&0&0&0&0&0&0&0&0\\
					0&0&2&0&0&0&0&0&0&0\\
					0&0&0&2&0&0&0&0&0&0\\
					0&0&0&0&2&0&0&0&0&0\\
					0&0&0&0&0&6&0&0&0&0\\
					4&0&0&0&0&0&2&0&0&0\\
					0&20&0&0&0&0&0&10&0&0\\
					0&738&0&0&0&468&0&36&18&0\\
					0&194450&0&0&0&24690&0&192000&9240&40
				\end{pmatrix}.
			\end{equation*}
			Hence these signatures form an {\rm ISS} with a lower triangular {\rm ISM} of $W({\rm H_3})$.
		\end{example}
		\begin{example}
			For the symmetric group $S_n=W({\rm A_{n-1}})$, we display permutation $\pi\in S_n$ by cycle notation
			\begin{equation*}
				\pi=(i_1,i_2,\ldots,i_l)\cdots(i_m,i_{m+1},\ldots,i_n),\quad i_j\in[n].
			\end{equation*}
			Then Coxeter generating set $T=\{(1,2),\ldots,(n-1,n)\}$.
			\par
			A {\it partition} of $n$ is a sequence
			\begin{equation*}
				\mu=(p_1,\ldots,p_k)
			\end{equation*}
			where the $p_i$ are weakly decreasing and $\sum_{i=1}^kp_i=n$, written by $\mu\vdash n$. Suppose $\pi$ consists of $k$ cycles whose lengths are $q_1,\ldots,q_k$ respectively where $q_1\geq\cdots\geq q_k$, the {\it type} of $\pi$ is defined as $(q_1,\ldots,q_k)$, which is a partition of $n$. Let
			\begin{equation*}
				C^{\mu}\coloneqq\{\pi:{\it type}\ \pi=\mu\}.
			\end{equation*}
			Let $\mu_1,\ldots,\mu_r$ be all partition of $n$, then all the conjugacy classes of $S_n$ are $C^{\mu_1},\ldots,C^{\mu_r}$. For any $\mu_j=(p_{j1},\ldots,p_{jk_j})\vdash n$, signature
			\begin{equation*}
				\alpha(\mu_j)\coloneqq(\underbrace{1,\ldots,1}_{p_{j1}-1},0,\underbrace{1,\ldots,1}_{p_{j2}-1},0,\ldots,0,\underbrace{1,\ldots,1}_{p_{jk_j}-1}).
			\end{equation*}
			It follows that $E_{\alpha(\mu_j)}\subseteq\phi^{-1}(C^{\mu_j})$ and $|E_{\alpha(\mu_j)}|=(n-k_j)!$. Hence $\alpha(\mu_1),\ldots, \\ \alpha(\mu_r)$ is an {\rm ISS} with a lower triangular {\rm ISM}
			\begin{equation*}
				diag((n-k_1)!,\ldots,(n-k_r)!).
			\end{equation*}
			\par
			Each partition $\mu\vdash n$ corresponds to a irreducible representation $\rho^{\mu}$ named {\it Young's natural representation}, see \cite{S}. All the $\rho^{\mu}$ form a complete list of irreducible representations of $S_n$. The characters and the characteristic polynomials of Young's natural representations for $S_3,S_4$ are given in the following tables.
			
			\begin{table}[H]
				\renewcommand{\arraystretch}{1.3}
				\centering
				\caption{Characters and characteristic polynomials for $S_3$}
				\scalebox{1}{
					\begin{NiceTabular}{|m{0.08\textwidth}<{\centering}|m{0.08\textwidth}<{\centering}|m{0.07\textwidth}<{\centering}|m{0.06\textwidth}<{\centering}|m{0.25\textwidth}|}
						\hline
						\multirow{2}{*}{$\rho^{\mu}$} & \multicolumn{3}{c}{$\chi$} & \multirow{2}{*}{$d(T,\rho^{\mu})$}\\
						\cline{2-4}
						& $C^{(1,1,1)}$ & $C^{(2,1)}$ & $C^{(3)}$ &\\
						\hline
						$\rho^{(3)}$ & 1 & 1 & 1 & $x_0+x_1+x_2$\\
						\hline
						$\rho^{(2,1)}$ & 2 & 0 & -1 & $x_0^2-x_1^2-x_2^2+x_1x_2$\\
						\hline
						$\rho^{(1,1,1)}$ & 1 & -1 & 1 & $x_0-x_1-x_2$\\
						\hline
					\end{NiceTabular}
				}
			\end{table}
			\begin{table}[H]
				\renewcommand{\arraystretch}{1.4}
				\centering
				\caption{Characters and characteristic polynomials for $S_4$}
				\scalebox{1}{
					\begin{NiceTabular}{|m{0.035\textwidth}<{\centering}|m{0.035\textwidth}<{\centering}|m{0.035\textwidth}<{\centering}|m{0.035\textwidth}<{\centering}|m{0.035\textwidth}<{\centering}|m{0.035\textwidth}<{\centering}|m{0.55\textwidth}|}
						\hline
						\multirow{2}{*}{$\rho^{\mu}$} & \multicolumn{5}{c}{$\chi$} & \multirow{2}{*}{$d(T,\rho^{\mu})$}\\
						\cline{2-6}
						& $C^{\mu_5}$ & $C^{\mu_4}$ & $C^{\mu_3}$ & $C^{\mu_2}$ & $C^{\mu_1}$ &\\
						\hline
						$\rho^{\mu_1}$ & 1 & 1 & 1 & 1 & 1 & $x_0+x_1+x_2+x_3$\\
						\hline
						$\rho^{\mu_2}$ & 3 & 1 & -1 & 0 & -1 & $x_0^3-x_1^3-x_2^3-x_3^3+x_0^2x_1+x_0^2x_2+x_0^2x_3-x_0x_1^2+x_1^2x_3-x_0x_2^2-x_0x_3^2+x_1x_3^2+x_0x_1x_2+2x_0x_1x_3+x_0x_2x_3$\\
						\hline
						$\rho^{\mu_3}$ & 2 & 0 & 2 & -1 & 0 & $x_0^2-x_1^2-x_2^2-x_3^2+x_1x_2+x_2x_3-2x_1x_3$\\
						\hline
						$\rho^{\mu_4}$ & 3 & -1 & -1 & 0 & 1 & $x_0^3+x_1^3+x_2^3+x_3^3-x_0^2x_1-x_0^2x_2-x_0^2x_3-x_0x_1^2-x_1^2x_3-x_0x_2^2-x_0x_3^2-x_1x_3^2+x_0x_1x_2+2x_0x_1x_3+x_0x_2x_3$\\
						\hline
						$\rho^{\mu_5}$ & 1 & -1 & 1 & 1 & -1 & $x_0-x_1-x_2-x_3$\\
						\hline
						\multicolumn{7}{l}{$\mu_1=(4)$, $\mu_2=(3,1)$, $\mu_3=(2,2)$, $\mu_4=(2,1,1)$, $\mu_5=(1,1,1,1)$.}\\
						\hline
					\end{NiceTabular}
				}
			\end{table}
			As an application of Theorem \refeq{MT}, the characteristic polynomials of representations can help us calculate the decomposition of the restricted representation $\rho\downarrow_{S_m}^{S_n}$ for $m<n$. Let $\rho\in\mathcal{R}(S_n)$ and $T_n$ be the Coxeter generating set of $S_n$. It follows that
			\begin{equation}\label{snsmeq}
				d(T_m,\rho\big\downarrow_{S_m}^{S_n})=d(T_n,\rho)(x_0,x_1,\ldots,x_{m-1},0,\ldots,0).
			\end{equation}
			So the decomposition of $\rho\downarrow_{S_m}^{S_n}$ corresponds to the decomposition of the right side of \eqref{snsmeq}. For instance,
			\begin{align*}
				&d(T_3,\rho^{(3,1)}\big\downarrow_{S_3}^{S_4})\\
				=&d(T_4,\rho^{(3,1)})(x_0,x_1,x_2,0)\\
				=&x_0^3-x_1^3-x_2^3+x_0^2x_1+x_0^2x_2-x_0x_1^2-x_0x_2^2+x_0x_1x_2\\
				=&(x_0+x_1+x_2)(x_0^2-x_1^2-x_2^2+x_1x_2)\\
				=&d(T_3,\rho^{(3)})d(T_3,\rho^{(2,1)}).
			\end{align*}
			So we have $\rho^{(3,1)}\downarrow_{S_3}^{S_4}\cong\rho^{(3)}\oplus\rho^{(2,1)}$, which can also be obtained by {\it Branching Rule}, see \cite{S}.
		\end{example}
			
		\begin{remark}
			In this paper, by studying the existence of the independent signature sequence, we generalize the main result of \cite{CST} to any finite Coxeter groups. A semigroup isomorphism between representations and their characteristic polynomials for finite Coxeter groups is established. We conjecture such sequence may exist for a larger category of finitely generated groups so that the result can be further generalized. On the other hand, there are many interesting topics about these polynomials, such as whether the characteristic polynomial for an irreducible representation is irreducible, how to calculate the characters from the polynomials directly. Therefore, we need more efforts to work on these topics.
		\end{remark}

		Shoumin Liu\\
		Email: s.liu@sdu.edu.cn\\
		School of Mathematics, Shandong University\\
		Shanda Nanlu 27, Jinan, \\
		Shandong Province, China\\
		Postcode: 250100\\
		Yuxiang Wang\\
		Email: quince773@163.com\\
		School of Mathematics, Shandong University\\
		Shanda Nanlu 27, Jinan, \\
		Shandong Province, China\\
		Postcode: 250100

\end{document}